\documentclass[12pt]{amsart}
\usepackage[margin=3cm]{geometry}
\usepackage{tikz,amsthm,amsmath,comment,amsfonts}
\usepackage{xcolor}
\usepackage{textcomp}
\usepackage{stmaryrd}
\usepackage{proofsteps}
\usepackage{connectives}
\usepackage{bbm}
\usepackage{float}
\usepackage{times}
\usepackage{enumitem}

\usepackage{setspace}
\def\etc{{\it etc}}
\def\LaTeX{L\kern-.36em\raise.3ex\hbox{a}\kern-.15em
    T\kern-.1667em\lower.7ex\hbox{E}\kern-.125emX}

\usepackage{url}
\usepackage{textcomp}
\usepackage{graphics}
\usepackage{epsfig}
\usepackage{times}
\usepackage{amssymb}
\usepackage[colorlinks]{hyperref}
\usepackage[ps,curve,2cell]{xypic}
\UseAllTwocells
\usepackage[active]{srcltx}
\usepackage{xy}
 \xyoption{v2}
\usepackage{multirow}
\usepackage{pdfsync}
\usepackage[active]{srcltx}

\usepackage[colorlinks]{hyperref}

\newtheorem{definition}{Definition}[section]
\makeatletter
  \def\th@definition{
  \thm@headfont{\bfseries} 
  \thm@notefont{\mdseries\itshape} 
  \upshape
}
\makeatother
\theoremstyle{definition}
\newtheorem{example}{Example}[section]
\newtheorem{theorem}{Theorem}[section]
\newtheorem{proposition}{Proposition}[section]
\newtheorem{fact}{Fact}[section]
\newtheorem{lemma}{Lemma}[section]

\newtheorem{corollary}{Corollary}[section]
\newtheorem{notation}{Notation}[section]

\newcommand{\ol}[1]{\overline{#1}}
\newcommand{\wt}[1]{\widetilde{#1}}
\newcommand{\naturalto}{%
  \mathrel{\vbox{\offinterlineskip
    \mathsurround=0pt
    \ialign{\hfil##\hfil\cr
      \normalfont\scalebox{1.2}{.}\cr
      $\longrightarrow$\cr}
  }}%
}

\usepackage{enumitem}
\setlist[enumerate, 1]{
  leftmargin=\parindent,
  parsep=.5\parsep,
  itemsep=0ex
}

\setlist[itemize, 1]{
  leftmargin=\parindent,
  parsep=.5\parsep,
  itemsep=0ex
}

\def\Nm{\mathit{Nm}}

\def\sem#1{[\![ #1 ] \!]}

\newcommand{\co}{\,\colon\;}
\newcommand{\ra}{\rightarrow}
\newcommand{\Ra}{\Rightarrow}

\newcommand{\op}{\mathrm{op}}

\newcommand{\impl}{\Rightarrow}

\newcommand{\snvsp}{\vspace{-1ex}}

\newcommand{\Nom}{\mathrm{Nom}}

\def\sen#1{#1\mbox{-}\mathrm{sen}}

\newcommand{\Sign}{\mathit{Sign}}
\newcommand{\Sen}{\mathit{Sen}}

\newcommand{\Mod}{\mathit{Mod}}

\def\B{\mathcal{B}}

\def\Set{\mathbf{Set}}
\def\CAT{\mathbf{C\!A\!T}}

\def\INS{\mathbf{IN\!S}}
\def\SINS{\mathbf{S\!IN\!S}}

\def\I{\mathcal{I}}
\def\SI{\mathcal{S}}
\def\FOL{{F\!O\!L}}

\def\PL{{P\!L}}

\def\OFOL{{O\!F\!O\!L}}
\def\MOFOL{{M\!O\!F\!O\!L}}
\def\HOFOL{{H\!O\!F\!O\!L}}
\def\HMOFOL{{H\!M\!O\!F\!O\!L}}
\def\MPL{{M\!P\!L}}
\def\MFOL{{M\!F\!O\!L}}
\def\HPL{{H\!P\!L}}
\def\HHPL{{H\!H\!P\!L}}
\def\HHHPL{{H\!H\!H\!P\!L}}
\def\HFOL{{H\!F\!O\!L}}
\def\HHFOL{{H\!H\!F\!O\!L}}
\def\MMPL{{M\!M\!P\!L}}
\def\MMFOL{{M\!M\!F\!O\!L}}
\def\MHPL{{M\!H\!P\!L}}
\def\MHFOL{{M\!H\!F\!O\!L}}
\def\REL{{R\!E\!L}}

\def\RELC{{R\!E\!L\!C}}
\def\BREL{{B\!R\!E\!L}}
\def\BRELC{{B\!R\!E\!L\!C}}
\def\SETC{{S\!E\!T\!C}}
\def\SET{\mathit{SET}}

\def\opp{\mathrm{op}}

\DeclareMathAlphabet{\mathbb}{U}{msb}{m}{n}
\DeclareSymbolFont{ams}{U}{msa}{m}{n}
\DeclareSymbolFontAlphabet{\mathams}{ams}
\DeclareMathSymbol{\filter}{\mathams}{ams}{22}

\usepackage{xcolor}

\begin{document}

\title{Decompositions of Stratified Institutions} 
\author{R\u{a}zvan Diaconescu}
\address{Simion Stoilow Institute of Mathematics of the Romanian
  Academy}
\email{Razvan.Diaconescu@ymail.com}

\date{}


\begin{abstract}
The theory of \emph{stratified institutions} is a general axiomatic
approach to model theories where the satisfaction is parameterised by
states of the models.
In this paper we further develop this theory by introducing a new
technique for representing stratified institutions which is based on
projecting to such simpler structures.
On the one hand this can be used for developing general results
applicable to a wide variety of already existing model theories with
states, such as those based on some form of Kripke semantics.
On the other hand this may serve as a template for defining new such
model theories.
In this paper we emphasise on the former application of this technique
by developing general results on model amalgamation and on existence
diagrams for stratified institutions.
These are two most useful properties to have in institution theoretic
model theory. 
\end{abstract}

\maketitle

\section{Introduction}

\subsection{Stratified institutions}

Institution theory is a general axiomatic approach to model theory
that has been originally introduced in computing science by Goguen and
Burstall \cite{ins}. 
In institution theory all three components of logical systems --
namely the syntax, the semantics, and the satisfaction relation
between them -- are treated fully abstractly by relying heavily on
category theory.
This approach has impacted significantly both theoretical computing
science \cite{sannella-tarlecki-book} and model theory as such
\cite{iimt}.\footnote{Both mentioned monographs rather reflect the
  stage of development of institution theory and its applications at
  the moment they were published or even before that. In the meantime a
  lot of additional important developments have already taken place.
  At this moment the literature on institution theory and around, that
  has been developed over the course of four decades or so is rather
  vast.} 
In computing science the concept of institution has emerged as the
most fundamental mathematical structure of logic-based formal
specifications, a great deal of theory being developed at the general
level of abstract institutions.
In model theory the institution theoretic approach meant an axiomatic
driven redesign of core parts of model theory at a new level of
generality -- namely that of abstract institutions -- independently of
any concrete logical system.
Moreover, there is a strong interdependency between the two lines of
developments.

The institution theoretic approach to model theory has also been
refined in order to address directly some important non-classical
model theoretic aspects.
One such direction is motivated by \emph{models with states}, which
appear in myriad forms in computing science and logic.
A typical important class of examples is given by the Kripke semantics
(of modal logics) which itself comes in a wide variety of forms.
Moreover, the concept of model with states goes beyond Kripke
semantics, at least in its conventional acceptations.
The institution theory answer to this is given by the theory of
\emph{stratified institutions} introduced in \cite{modalins,strat}
and further developed or invoked in works such as
\cite{KripkeStrat,aiguier-bloch2019,gaina-acm}, etc. 

\subsection{New contributions}

Our work is part of a broader effort to further develop the theory of
stratified institutions in several directions that are important
either from a model theory as such or a computing science
perspective. 
Our developments consist of
\begin{itemize}

\item a new technique for representing stratified institutions, 

\item on the basis of the above mentioned technique, results
  establishing a couple of crucial model theoretic properties
  for stratified institutions, 

\end{itemize}
In what follows we discuss these in more detail. 

\subsubsection{Decompositions of stratified institutions}

Historically there have been two major approaches to Kripke semantics
within institution theory:
\begin{enumerate}

\item The approach introduced in \cite{ks} and then used in
  \cite{HybridIns,QVHybrid,EncHybrid,IntroducingH,gaina2015}, etc.,
  that considers Kripke semantics as a two-layered concept.
  A base layer consists of unspecific structures such as the
  interpretations of sorts (types), function, relation (predicate)
  symbols, etc.
  An upper layer consists of the structures specific to Kripke
  semantics and, at the syntactic level, of modalities.
  In that approach the base layer can be considered as a parameter and
  treated fully implicitly as an abstract institution while the upper
  layer considers explicit Kripke structures and modalities that are
  parameterised by the base layer. 

\item The approach of the stratified institutions that is fully
  abstract without any explicit Kripke structures and modalities. 
  
\end{enumerate}
The drawback of the former approach is precisely its rigid commitment
to a specific common concept of Kripke semantics and modal syntax.
Each time one deals with different such concepts, or even goes beyond
Kripke semantics, one has to reconstruct this upper level  and
redevelop most of the theory often by repeating the same ideas. 
On the other hand, due to its high abstraction level, the latter
approach is free of such issues and supports a full top down
development process where concepts are introduced axiomatically on a
by-need basis.
A typical example of this methodology comes from \cite{KripkeStrat}
where Kripke semantics and modalities are not assumed explicitly but
are treated implicitly in a fully modular and axiomatic manner.

In order to retain some of the benefits of the two-layered approach,
such as the hierarchical shape of the respective model theories that
at the bottom are based on a concept of possible worlds, here we take
a step further in this methodology by introducing a  concept of
decomposition of a stratified institution. 
In brief, we associate to a stratified institution a couple of
abstract projections to other stratified institutions that in examples
correspond to the two layers discussed above.
But now the projection corresponding to the upper layer is fully
abstract.
Most examples / applications of stratified institutions can be
presented as decomposed stratified institutions in a meaningful way. 

\subsubsection{Model amalgamations and diagrams}

The institution theoretic analysis of model theory has established
model amalgamation and the method of diagrams as the most pervasive
properties in model theory.
While in classical concrete model theory the prominent role of the
latter is recognised as such, model amalgamation has a rather implicit
role.
In fact it is the merit of the institution theoretic approach to model
theory to bring model amalgamation to surface and reveal its
importance. 
With respect to diagrams, this concept got a fully abstract
institution theoretic formulation in \cite{edins}.
The literature on institution theoretic model theory abounds of
situations supporting the claims about the role of the two properties
and many of these can be found in \cite{iimt}.
At the general level of bare abstract stratified institutions both
properties have to be assumed and then established only at the level
of the concrete examples.
By the decomposition technique discussed above we are able to actually
establish these properties at the level of abstract stratified
institutions from the corresponding properties of the two
``components''.
The value of these results reside in the fact that, on the one hand
they define classes of abstract stratified institutions that admit
these properties, and on the other hand they provide an easy way to
establish them in concrete situations because the problem is reduced
to the two components where things are much simplified.  

\subsection{Summary of contents}

This article is structured as follows:
\begin{enumerate}

\item In a preliminary section we review some basic concepts from the
  common institution theory and also from stratified institution
  theory.

\item In the next section we define the concept of decomposition of
  stratified institutions.

\item In a section dedicated to model amalgamation we explore two
  different concept of model amalgamation that are relevant for
  stratified institutions.
  Then we develop a general result on the existence of model
  amalgamation for decomposed stratified institutions.

\item In a section dedicated to diagrams we develop a construction of
  diagrams in decomposed stratified institutions.

\end{enumerate}

\section{Preliminaries}

In this section we recall from the literature some category and
institution theoretic notions that will be used in the paper.
However Section \ref{si-morphism} is an exception in the sense that it
introduces a new concept.
In order to enhance the readability of the paper by keeping it
reasonably self contained we deliberately take the choice of a
relatively extensive review of the needed concepts as well as of
relevant examples.

The contents of this section is as follows:
\begin{enumerate}

\item We fix some category theory notations and list the category
  theoretic concepts needed in order to study this work.

\item We recall the definition of institution and present briefly a
  couple of the most common examples of model theories captured as
  institutions.
  We provide a list of `sub-institutions' of first-order logic that
  will be used in the paper.

\item We recall one of the two dual concepts of mappings of institutions,
  namely that of institution morphism, which is the relevant one for
  our work.  

\item We recall the definition of stratified institutions.

\item We present the general representation of stratified institutions
  as ordinary institutions that has been introduced in
  \cite{KripkeStrat}.
  This is a mere technical device.

\item We discuss a representative list of examples of stratified
  institutions.

\item We extend the concept of institution morphism from ordinary
  institution theory to stratified institutions.

\item By following \cite{KripkeStrat} we present an implicit concept
  of nominals in stratified institutions.
  
\end{enumerate}

\subsection{Category theory}

The mathematical structures in institution theory are category
theoretic.
We usually follow the terminology and notations of \cite{maclane98}
with some few notable exceptions.
One of them is the way we write compositions.
Thus we will use the diagrammatic notation for compositions of arrows
in categories, i.e. if $f \co A \ra B$ and $g\co B \ra C$ are arrows
then $f;g$ denotes their composition.
Let $\Set$ denote the category of sets, $\CAT$ denote the
``quasi-category'' of categories, $|\CAT|$ the collection of all
categories. 
We use $\Ra$ rather than  $\naturalto$ for natural transformations. 

The following category theory concepts are used in our work:
opposite (dual) of category, sub-category, functor, functor category,
natural transformation, lax natural transformation, comma category,
(direct) product, co-product, pushout, pullback, epimorphism (epi),
adjunction. 
All these belong to the somehow elementary level of category theory.
In general institution theory seldom requires category theory beyond
that level.
Familiarity with these concepts is a requirement for being able to
follow this work.

\subsection{Institutions}

The seminal mathematical structure of institution theory is given in
Definition \ref{ins-dfn} below from \cite{ins}. 

\begin{definition}[Institution]\label{ins-dfn}
An  \emph{institution} $\I = 
\big(\Sign^{\I}, \Sen^{\I}, \Mod^{\I}, \models^{\I}\big)$ consists of 
\begin{itemize}

\item a category $\Sign^{\I}$ whose objects are called
  \emph{signatures},

\item a sentence functor $\Sen^{\I} \co \Sign^{\I} \ra \Set$
  defining for each signature a set whose elements are called
  \emph{sentences} over that signature and defining for each signature
  morphism a \emph{sentence translation} function, 

\item a model functor $\Mod^{\I} \co (\Sign^{\I})^{\op} \ra \CAT$
  defining for each signature $\Sigma$ the category
  $\Mod^{\I}(\Sigma)$ of \emph{$\Sigma$-models} and $\Sigma$-model
  homomorphisms, and for each signature morphism $\varphi$ the
  \emph{reduct} functor $\Mod^{\I}(\varphi)$,  

\item for every signature $\Sigma$, a binary 
  \emph{$\Sigma$-satisfaction relation}
  $\models^{\I}_{\Sigma} \subseteq |\Mod^{\I} (\Sigma)|
  \times \Sen^{\I} (\Sigma)$, 

\end{itemize}
such that for each morphism 
$\varphi\co\Sigma \rightarrow \Sigma' \in \Sign^{\I}$, 
the \emph{Satisfaction Condition}
\begin{equation}
M'\models^{\I}_{\Sigma'} \Sen^{\I}(\varphi) \rho \text{ if and only if  }
\Mod^{\I}(\varphi) M' \models^{\I}_\Sigma \rho
\end{equation}
holds for each $M'\in |\Mod^{\I} (\Sigma')|$ and $\rho \in \Sen^{\I} (\Sigma)$.
\[
\xymatrix{
    \Sigma \ar[d]_{\varphi} & \big|\Mod^{\I}(\Sigma)\big|
    \ar@{-}[r]^-{\models^{\I}_{\Sigma}} & 
    \Sen^{\I}(\Sigma) \ar[d]^{\Sen^{\I}(\varphi)} \\
    \Sigma' & \big| \Mod^{\I}(\Sigma')\big| \ar[u]^{\Mod^{\I}(\varphi)} 
    \ar@{-}[r]_-{\models^{\I}_{\Sigma'}} & \Sen^{\I}(\Sigma')
  }
\]
\end{definition}

\begin{notation}
We may omit the superscripts or subscripts from the notations of the
components of institutions when there is no risk of ambiguity. 
For example, if the considered institution and signature are clear,
we may denote $\models^{\I}_\Sigma$ just by $\models$. 
For $M = \Mod(\varphi)(M')$, we say that $M$ is the
\emph{$\varphi$-reduct} of $M'$ and that $M'$ is a
\emph{$\varphi$-expansion} of $M$.
Moreover in order to further simplify notations we may sometimes
denote $\Sen(\varphi) \rho$ by $\varphi \rho$ and
$\Mod(\varphi) M'$ by $\varphi M'$. 
\end{notation} 











The literature  shows myriads of logical systems from computing or
from mathematical logic captured 
as institutions.
Many of these are collected in \cite{iimt,sannella-tarlecki-book}.
In fact, an informal thesis underlying institution 
theory is that any `logic' may be captured by the above
definition. While this should be taken with a grain of salt, it
certainly applies to any logical system based on satisfaction between
sentences and models of any kind.
The institutions introduced in the following couple of examples will
be used intensively in the paper in various ways. 

\begin{example}[Propositional logic ($\PL$)]
\begin{rm}
This is defined as follows.
$\Sign^{\PL} = \Set$, for any set $P$, $\Sen(P)$ is generated by the
grammar
\[
S ::= P \mid S \wedge S \mid \neg S 
\]
and $\Mod^{\PL}(P) = (2^P,\subseteq)$.
For any function $\varphi \co P \ra P'$, $\Sen^{\PL}(\varphi)$
replaces the each element $p\in P$ that occur in a sentence $\rho$ by
$\varphi(p)$, and $\Mod^{\PL}(\varphi)(M') = \varphi;M$ for each
$M'\in 2^{P'}$. 
For any $P$-model $M \subseteq P$ and $\rho\in \Sen^{\PL}(P)$,
$M\models\rho$ is defined by induction on the structure of $\rho$ by
$(M \models p) = (p\in M)$, 
$(M \models \rho_1 \wedge \rho_2) = (M \models \rho_1) \wedge (M
\models \rho_2)$ and 
$(M \models \neg\rho) = \neg(M \models \rho)$.  
\end{rm}
\end{example}

\begin{example}[First order logic ($\FOL$)]
\begin{rm}
For reasons of simplicity of notation, our presentation of first order
logic as institution considers only its single sorted, without
equality, variant.  
A detailed presentation of full many sorted first order logic with
equality as institution may be found in numerous works in the
literature (e.g. \cite{iimt}, etc.).
 
The $\FOL$ signatures are pairs 
$(F=(F_n)_{n\in   \omega},P=(P_n)_{n\in \omega})$ where 
$F_n$ and $P_n$ are sets of function symbols and predicate symbols,
respectively, of arity $n$. 
Signature morphisms $\varphi \co (F,P) \ra (F',P')$ are tuples
$(\varphi^{f}=(\varphi^{f}_n)_{n\in
  \omega},\varphi^{p}=(\varphi^{p}_n)_{n\in\omega})$ such that  
$\varphi^{f}_n \co F_n \ra F'_n$ and $\varphi^{p}_n \co P_n \ra
P'_n$. 
Thus $\Sign^{\FOL}=\Set^\omega \times \Set^\omega$. 

For any $\FOL$-signature $(F,P)$, the set $S$ of the $(F,P)$-sentences
is generated by the grammar: 
\begin{equation}\label{fol-grammar}
S ::= \pi(t_1,\dots,t_n) \mid S \wedge S \mid \neg S \mid (\exists x)S'
\end{equation}
where $\pi(t_1,\dots,t_n)$ are the atoms with $\pi\in P_n$ and
$t_1,\dots,t_n$ being terms formed with function symbols from $F$, and
where $S'$ denotes the set of $(F+x,P)$-sentences with $F+x$ denoting
the family of function symbols obtained by adding the single variable
$x$ to $F_0$.

An $(F,P)$-model $M$ is a triple 
\[
M=(|M|,\{ M_\sigma \co |M|^n \ra |M| \mid \sigma\in F_n, n\in\omega \},
\{ M_\pi \subseteq |M|^n \mid \pi\in P_n, n\in\omega \}).
\]
where $|M|$ is a set called the \emph{carrier of $M$}. 
An $(F,P)$-model homomorphism $h\co M \ra N$ is a function $|M| \ra
|N|$ such that $h(M_\sigma (x_1,\dots,x_n)) = N_\sigma
(h(x_1),\dots,h(x_n))$ for any $\sigma \in F_n$ and $h(M_\pi)
\subseteq N_\pi$ for each $\pi\in P_n$. 

The satisfaction relation $M \models^{\FOL}_{(F,P)} \rho$ is the usual
Tarskian style satisfaction defined on induction on the structure of
the sentence $\rho$. 

Given a signature morphism $\varphi \co (F,P) \ra (F',P')$, the
induced sentence translation $\Sen^{\FOL}(\varphi)$ just replaces the
  symbols of any $(F,P)$-sentence with symbols from $(F',P')$
  according $\varphi$, and the induced model reduct
  $\Mod^{\FOL}(\varphi)(M')$ leaves the carrier set as it is and for
  any $x$ function or predicate symbol of $(F,P)$, it interprets $x$ as
  $M'_{\varphi(x)}$.  

In what follows we shall also consider the following parts (or
`sub-institutions') of $\FOL$ that are determined by restricting the
$\FOL$ signatures as follows:
\begin{itemize}

\item $\REL$: no function symbols (hence $\Sign^{\REL} \cong
  \Set^\omega$);

\item $\RELC$: no function symbols of arity greater than 0; 

\item $\BREL$: no function symbols and only one binary predicate symbol
  $\lambda$   (hence $\Sign^\BREL \cong \{ \lambda \}$);  

\item $\SETC$: no predicate symbols and no function symbols of arity
  greater than $0$ (hence  $\Sign^{\SETC} \cong \Set$); 

\item $\BRELC$: one binary predicate symbol and no function symbols of arity
  greater than $0$ (hence $\Sign^{\BRELC} \cong \Set$); 


\end{itemize}
\end{rm}
\end{example}





\subsection{Institution morphisms}

From the perspective of the mathematical structure, institution
morphisms \cite{ins} are just `homomorphisms of institutions'. 
So they are mappings between institutions that preserve the
mathematical structure of institutions.  

\begin{definition}[Morphism of institutions]
  \label{institution-morphism}
Given two institutions \\
$\I_i = (\Sign_i, \Sen_i, \Mod_i , \models_i)$,
with $i \in \{1, 2\}$, an \emph{institution morphism} \\
$(\Phi, \alpha, \beta) \co \I_2 \ra \I_1$ consists of 
\begin{itemize}

\item a signature functor $\Phi \co \Sign_2 \ra \Sign_1$,

\item a natural transformation $\alpha \co \Phi;\Sen_1 \Ra \Sen_2$,
  and 

\item a natural transformation $\beta\co \Mod_2 \Ra \Phi^{\op} ;
  \Mod_1$ 

\end{itemize}
such that the following Satisfaction Condition holds for  
any $\I_2$-signature $\Sigma_2$, $\Sigma_2$-model $M_2$ and
$\Phi(\Sigma_2)$-sentence $\rho$:
\[
M_2  \ \models_2 \ \alpha_{\Sigma_2} \rho \mbox{ \ if and only if \ }
\beta_{\Sigma_2} M_2 \ \models_1 \ \rho.
\]
\end{definition}

There is a dual notion of `homomorphism of institutions' in which
the direction of the sentence translations is reversed
\cite{jm-granada89,tarlecki95,tarlecki98,mossakowski96,goguen-rosu2000}.  
These are currently called \emph{comorphism} and although they bear
symmetry with morphisms their usage is very different. 
While institution morphisms have a projection feeling the comorphism
have an embedding feeling. 
However the latter are also used for encoding `more complex'
institutions to `simpler' institution by using the technique of
institutional theories (more details on that may be found in
\cite{iimt}). 

For examples we refer to \cite{iimt}.
Under a straightforward concept of composition, defined component-wise
on the three components (see \cite{iimt}), we get a category with
the institutions as objects and the institution morphisms as arrows.

\subsection{Stratified institutions}

Informally, the main idea behind the concept of stratified institution
as introduced in \cite{strat} is to enhance the concept of institution
with `states' for the models.
Thus each model $M$ comes equipped with a \emph{set} $\sem{M}$. 
A typical example is given by the Kripke models, where $\sem{M}$ is
the set of the possible worlds in the Kripke structure $M$.
However this is not the only possibility for models with states. 

The following definition has been given in \cite{KripkeStrat} and
represents an important upgrade of the original definition from
\cite{strat}, the main reason being to make the definition of
stratified institutions really usable for doing in-depth model theory.  
Independently another upgrade has been proposed in
\cite{aiguier-bloch2019}; however there is a strong convergence
between the two upgrades. 

\begin{definition}[Stratified institution]\label{strat-dfn}
A {\em stratified institution} $\SI$ is a tuple 
\[
(\Sign^\SI, \Sen^\SI, \Mod^\SI,\sem{\_}^\SI,\models^\SI)
\]
consisting of: 
\begin{itemize}

\item[--] a category $\Sign^\SI$ of signatures,

\item[--] a sentence functor $\Sen^\SI \co  \Sign^\SI \rightarrow \Set$;

\item[--] a model functor $\Mod^\SI \co (\Sign^\SI)^{\op} \rightarrow \CAT$; 

\item[--] a ``stratification'' lax natural transformation $\sem{\_}^\SI\co
  \Mod^\SI \Ra \SET$, where $\SET\co \Sign^\SI \ra \CAT$ is a functor
  mapping each signature to $\Set$; and 

\item[--] a satisfaction relation between models and sentences which is
parameterized by model states,
$M \ (\models^\SI)^w_\Sigma \ \rho$ where $w \in \sem{M}^\SI_\Sigma$
such that  
\begin{equation}\label{strat-sat-cond-eq}
\Mod^\SI(\varphi) M' \ \ (\models^\SI)^{\sem{M'}_\varphi w}_\Sigma \ \ \rho
\mbox{ \ if and only if \ } 
M' \ \ (\models^\SI)^w_{\Sigma'} \ \ \Sen^\SI(\varphi) \rho
\end{equation}
holds for any signature morphism $\varphi \co \Sigma \ra \Sigma'$,
$\Sigma'$-model $M'$, $w\in \sem{M'}^\SI_{\Sigma'}$, and
$\Sigma$-sentence $\rho$. 
\end{itemize}
Like for ordinary institutions, when appropriate we shall also use
simplified notations without superscripts or subscripts that are clear
from the context.  
\end{definition}

The lax natural transformation property of $\sem{\_}$ is depicted in
the diagram below
\[
\xymatrix @C+2em {
\Sigma''  &  
  \Mod(\Sigma'') \ar[r]^{\sem{\_}_{\Sigma''}} \ar[d]_{\Mod(\varphi')}
  \ar@/^.9pc/[dr]_{\quad} 
  \ar@/_.7pc/[dr]
   &
  \Set
  \ar@{}[ld]^(.35){}="a"^(.62){}="b" 
  \ar@{=>} "a";"b"^{ \ \ \sem{\_}_{\varphi'}}  
  \ar[d]^{=}
     \\
\Sigma' \ar[u]_{\varphi'} &  
  \Mod(\Sigma') \ar[d]_{\Mod(\varphi)} \ar[r]|{\sem{\_}_{\Sigma'}}  
  \ar@/^.9pc/[dr]_{\quad} 
  \ar@/_.7pc/[dr] & 
  \Set  \ar[d]^{=} 
  \ar@{}[ld]^(.35){}="c"^(.62){}="d" 
  \ar@{=>} "c";"d"^{ \ \ \sem{\_}_{\varphi}}  
&  \\
\Sigma \ar[u]_{\varphi} &  
  \Mod(\Sigma) \ar[r]_{\sem{\_}_{\Sigma}} & \Set
}
\]
with the following compositionality property for each $\Sigma''$-model
$M''$: 
\begin{equation}\label{strat1}
\sem{M''}_{(\varphi;\varphi')} =
\sem{M''}_{\varphi'};\sem{\Mod(\varphi')(M'')}_\varphi.
\end{equation}
Moreover the natural transformation property of each
$\sem{\_}_\varphi$ is given by the commutativity of the following
diagram:
\begin{equation}\label{diag1} 
\xymatrix{
M' \ar[d]_{h'}  & & 
  \sem{M'}_{\Sigma'} \ar[r]^-{\sem{M'}_{\varphi}}
  \ar[d]_{\sem{h'}_{\Sigma'}} & 
  \sem{\Mod(\varphi)(M')}_{\Sigma} \ar[d]^{\sem{\Mod(\varphi)(h')}_\Sigma}
 \\
N' & & 
 \sem{N'}_{\Sigma'} \ar[r]_-{\sem{N'}_{\varphi}} & 
  \sem{\Mod(\varphi)(N')}_{\Sigma}
}
\end{equation}

\

The satisfaction relation can be presented as a natural transformation
$\models \co \Sen \Ra \sem{\Mod(\_) \ra \Set}$ where the functor 
$\sem{\Mod(\_) \ra \Set} \co \Sign \ra \Set$ is defined by 
\begin{itemize}

\item[--] for each signature $\Sigma\in |\Sign|$, 
$\sem{\Mod(\Sigma) \ra \Set}$ denotes the set of all the mappings 
$f \co |\Mod(\Sigma)| \ra \Set$ such that $f(M) \subseteq
\sem{M}_\Sigma$; and 

\item[--] for each signature morphism $\varphi \co \Sigma \ra \Sigma'$, 
  \[
    \sem{\Mod(\varphi) \ra \Set}(f)(M') = 
    \sem{M'}_\varphi^{-1}(f(\Mod(\varphi) M')).
  \]
  
\end{itemize}
A straightforward check reveals that the Satisfaction Condition
(\ref{strat-sat-cond-eq}) appears exactly as the naturality property
of $\models$: 
$$\xy
\xymatrix{
\Sigma \ar[d]_{\varphi}  & & 
  \Sen(\Sigma) \ar[r]^-{\models_\Sigma}
  \ar[d]_{\Sen(\varphi)} & 
  \sem{\Mod(\Sigma) \ra \Set} \ar[d]^{\sem{\Mod(\varphi) \ra \Set}}
 \\
\Sigma' & & 
  \Sen(\Sigma') \ar[r]_-{\models_{\Sigma'}} & 
  \sem{\Mod(\Sigma') \ra \Set}
}
\endxy$$

Ordinary institutions are the stratified institutions
for which $\sem{M}_\Sigma$ is always a singleton set. 
In Defintion~\ref{strat-dfn} we have removed the surjectivity
condition on $\sem{M'}_\varphi$ from the definition of the stratified
institutions of \cite{strat} and will rather make it  explicit when
necessary. 
This is motivated by the fact that most of the results developed do
not depend upon this condition which however holds in all examples
known by us.  
In fact in many of the examples $\sem{M'}_\varphi$ are even
identities, which makes $\sem{\_}$ a strict rather than lax natural
transformation.
In such cases the stratified institution itself is called a
\emph{strict stratified institution}. 
A notable exception, when $\sem{\_}$ is a proper lax natural
transformation is given by Example~\ref{ofol-ex}. 
Also the definition of stratified institution of \cite{strat} did not
introduce $\sem{\_}$ as a lax natural transformation, but rather as an
indexed family of mappings without much compositionality properties,
which was enough for the developments in \cite{strat}.  

The following very expected property does not follow from the axioms
of Definition~\ref{strat-dfn}, hence we impose it explicitly. 
It holds in all the examples discussed in this paper. 

\noindent
\textbf{Assumption:} In all considered stratified institutions the
satisfaction is preserved by model isomorphisms, i.e. for each
$\Sigma$-model isomorphism $h \co M \ra N$, each $w\in
\sem{M}_\Sigma$, and each $\Sigma$-sentence $\rho$, 
\[
M \models^w \rho \mbox{ \ if and only if \ }
N \models^{\sem{h} w} \rho.
\]

\subsection{Reducing stratified institutions to ordinary institutions}
\label{red-strat-sec}

The following construction from \cite{KripkeStrat} will be used
systematically in what follows for reducing stratified institution
theoretic concepts to ordinary institution theoretic concepts, and
consequently for reusing results from the latter to the former realm.  

\begin{fact}\label{sharp-institution-fact}
Each  stratified institution 
$\SI = (\Sign,\Sen,\Mod,\sem{\_},\models)$ determines the
following ordinary institution 
$\SI^\sharp = (\Sign,\Sen,\Mod^\sharp,\models^\sharp)$ (called the
\emph{local institution of $\SI$}) where
\begin{itemize}

\item[--] the objects of $\Mod^\sharp (\Sigma)$ are the pairs $(M,w)$
  such that $M\in |\Mod(\Sigma)|$ and $w\in \sem{M}_\Sigma$;  

\item[--] the $\Sigma$-homomorphisms  $(M,w)\ra (N,v)$ are the pairs
  $(h,w)$ such that $h \co M \ra N$ and $\sem{h}_\Sigma  w = v$; 

\item[--] for any signature morphism $\varphi\co \Sigma \ra \Sigma'$
  and any $\Sigma'$-model $(M',w')$
\[
\Mod^\sharp (\varphi)(M',w') = 
(\Mod(\varphi) M',\sem{M'}_\varphi w');
\] 

\item[--] for each $\Sigma$-model $M$, each $w\in \sem{M}_\Sigma$, and
  each $\rho\in \Sen(\Sigma)$
\begin{equation}\label{satcond}
((M,w) \models^\sharp_\Sigma \rho) = (M \models^w_\Sigma \rho).
\end{equation}
\end{itemize} 
\end{fact}

The preservation of $\models$ under model isomorphisms imply 
the preservation of $\models^\sharp$ under model isomorphisms.
This follows immediately by noting that $(h,w)$ is a model
isomorphism in $\SI^\sharp$ if and only if $h$ is a model isomorphism
in $\SI$. 

The following second interpretation of  stratified
institutions as ordinary institutions has been given in
\cite{strat}.
Note that unlike $\SI^\sharp$ above, $\SI^*$ below shares with $\SI$ the
model functor. 

\begin{definition}
For any stratified institution $\SI =
(\Sign,\Sen,\Mod,\sem{\_},\models)$ we say that \emph{$\sem{\_}$ is
surjective} when for each
signature morphism $\varphi \co \Sigma \ra \Sigma'$ and each
$\Sigma'$-model $M'$, $\sem{M'}_\varphi \co \sem{M'}_{\Sigma'} \ra
\sem{\Mod(\varphi) M'}_{\Sigma}$ is surjective.  
\end{definition}

\begin{fact}
Each  stratified institution $\SI =
(\Sign,\Sen,\Mod,\sem{\_},\models)$ with $\sem{\_}$ surjective
determines an (ordinary) institution $\SI^* =
(\Sign,\Sen,\Mod,\models^*)$ (called the \emph{global institution of
  $\SI$}) by defining   
\[
(M \models^*_\Sigma \rho) = 
\bigwedge \{ M \models^w_\Sigma \rho \mid w\in \sem{M}_\Sigma \}. 
\]
\end{fact} 


From now on whenever we invoke an institution $\SI^*$ we tacitly
assume that $\sem{\_}^{\SI}$ is surjective. 

The institutions $\SI^\sharp$ and $\SI^*$ represent generalizations of
the concepts of local and global satisfaction, respectively, from modal
logic (e.g. \cite{blackburn-rijke-venema2002}).
While $\SI^*$ ``forgets'' the stratification of $\SI$, $\SI^\sharp$
fully retains it (but in an implicit form).
This is the reason why $\SI^\sharp$ rather than $\SI^*$ can be used for
reflecting concepts and results from ordinary institution theory to
stratified institutions.
It is important to avoid a possible confusion regarding $\SI^\sharp$,
namely that through the flattening represented by the $\sharp$
construction stratified institution theory gets reduced to ordinary
institution theory.
This cannot be the case because although $\SI^\sharp$ being an
ordinary institution it has a particular character induced by the
stratified structure of $\SI$.
This means that many general institution theoretic concepts are not
refined enough to reflect properly the stratification aspects.  

\subsection{Concrete examples of stratified institutions}
\label{ex-sect}

Most of the examples presented below are various forms of modal logics
with Kripke semantics.
However a few of them go beyond the Kripke semantics. 
They can be found in greater detail in \cite{KripkeStrat}. 

\begin{example}[Modal propositional logic ($\MPL$)]\label{mpl-ex}
\begin{rm}
This is the most common form of modal logic
(e.g. \cite{blackburn-rijke-venema2002}, etc.).
 
Let $\Sign^{\MPL} = \Set$.
For any signature $P$, commonly referred to as `set of propositional
variables', the set of its sentences 
$\Sen^{\MPL}(P)$ is the set $S$ defined by the following grammar
\begin{equation}\label{mpl-grammar-equation}
S \ ::= \ P \mid S \wedge S \mid \neg S \mid \Diamond S
\end{equation}
A $P$-model is Kripke structure $(W,M)$ where 
\begin{itemize}

\item $W = (|W|,W_\lambda)$ consists of set (of `possible worlds')
  $|W|$ and an `accesibility' relation $W_\lambda \subseteq |W| \times
  |W|$; and 

\item $M \co |W| \ra 2^P$. 

\end{itemize}
A homomorphism $h \co (W,M) \ra (V,N)$ between Kripke structures is a
homomorphism of binary relations $h \co W \ra V$ (i.e. $h \co |W| \ra
|V|$ such that $h(W_\lambda) \subseteq V_\lambda$) and such that for
each $w\in |W|$, $M^w \subseteq N^{h(w)}$. 

The satisfaction of any $P$-sentence $\rho$ in a Kripke
structure $(W,M)$ at $w\in |W|$ is defined by recursion on the
structure of $\rho$: 
\begin{itemize}

\item $((W,M) \models_P^w \pi) = (\pi \in M^w)$;

\item $((W,M) \models_P^w \rho_1 \wedge \rho_2) = ((W,M) \models_P^w
  \rho_1) \wedge ((W,M) \models_P^w \rho_2)$; 

\item $((W,M) \models_P^w \neg \rho) = 
  \neg ((W,M) \models_P^w \rho)$; and 

\item $((W,M) \models_P^w \Diamond \rho) = 
\bigvee_{(w,w')\in W_\lambda}((W,M) \models_P^{w'} \rho)$. 

\end{itemize}
For any function $\varphi \co P \ra P'$ the $\varphi$-translation of a
$P$-sentence just replaces each $\pi\in P$ by $\varphi(\pi)$ and the
$\varphi$-reduct of a $P'$-structure $(W,M')$ is the $P$-structure
$(W,M)$ where for each $w\in |W|$, $M^w = \varphi;M'^w$. 

The stratification is defined by $\sem{(W,M)}_P = |W|$.   

Various `sub-institutions' of $\MPL$ are obtained by restricting the
semantics to particular classes of frames.
Important examples are ${\MPL}t$, ${\MPL}s4$, and ${\MPL}s5$ which are
obtained by restricting the frames $W$ to those which are
respectively, reflexive, preorder, or equivalence (see
e.g. \cite{blackburn-rijke-venema2002}).     
\end{rm}
\end{example} 

\begin{example}[First order modal logic ($\MFOL$)]
  \label{mfol-ex}
\begin{rm}
First order modal logic \cite{fitting-mendelsohn98} extends classical
first order logic with modalities in the same way 
propositional modal logic extends classical propositional logic.
However there are several variants that differ slightly in the
approach of the quantifications. 
Here we present a capture of one of the most common variants of first
order modal logic as a stratified institution. 

$\MFOL$ has the category of signatures of $\FOL$ but for the sentences
adds $S ::= \Diamond S$ to the $\FOL$ grammar (\ref{fol-grammar}). 
The $\MFOL$ $(F,P)$-models upgrade the $\MPL$ Kripke structures
$(W,M)$ to the first order situation by letting $M \co |W| \ra
|\Mod^{\FOL}(F,P)|$ such that the following sharing conditions hold:
for any $i,j \in |W|$, $|M^i| = |M^j|$ and also $M^i_x = M^j_x$ for
each constant $x\in F_0$. 
The concept of $\MFOL$-model homomorphism is also an upgrading of the
concept of $\FOL$-model homomorphism as follows: $h \co (W,M) \ra
(V,N)$ is pair $(h_0,h_1)$ where $h_0 \co W \ra V$ is a homomorphism
of binary relations (like in $\MPL$) and $h_1 \co M^w \ra N^{h_0 (w)}$
is an $(F,P)$-homomorphism of $\FOL$-models for each $w\in |W|$. 

The satisfaction $(W,M) \models^{\MFOL}_{(F,P)} \rho$ is defined by
recursion on the structure of $\rho$, like in $\MPL$ for $\wedge$,
$\neg$, and $\Diamond$, for the atoms the $\FOL$ satisfaction relation
is used, and for the quantifier case $(W,M) \models_{(F,P)} (\exists
x) \rho$ if and only if there is a valuation of $x$ into $|M|$ such
that $(W,M') \models_{(F+x,P)} \rho$ for the corresponding expansion 
$(W,M')$ of $(W,M)$ to $(F\!+\!x,P)$. 
(This makes sense because in any $\MFOL$ Kripke structure the
interpretations of the carriers and of the constants are shared.) 

The translation of sentences and the model reducts corresponding to an
$\MFOL$ signature morphism are obtained by the obvious blend of the
corresponding translations and reducts, respectively, in $\MPL$ and
$\FOL$. 

The stratification is like in $\MPL$, with $\sem{(W,M)}_{(F,P)} = |W|$.   

In the institution theory literature
(e.g. \cite{iimt,ks,HybridIns,QVHybrid}) first order modal logic is
often considered in a more general form in which the symbols that have
shared interpretations are `user defined' rather than being
`predefined' like here. 
In short this means that the signatures exhibit designated symbols
(sorts, function, or predicate) that are `rigid' in the sense that in
a given Kripke structure they share the same interpretations across
the possible worlds. 
For the single reason of making the reading easier we stick here with
a simpler variant that has constants and the single sort being
predefined as rigid. 
\end{rm}
\end{example}

\begin{example}[Hybrid logics ($\HPL$, $\HFOL$)]\label{hpl-ex}
\begin{rm}
Hybrid logics \cite{prior,blackburn2000} refine modal logics by adding
explicit syntax for the possible worlds. 
Our presentation of hybrid logics as stratified institutions is
related to the recent institution theoretic works on hybrid logics
\cite{HybridIns,QVHybrid}. 

The refinement of modal logics to hybrid ones is achieved by adding a
set component ($\Nom$) to the signatures for the so-called `nominals'
and by adding to the respective grammars 
\begin{equation}\label{hybrid-grammar}
S ::= \sen{i} \mid @_i S \mid (\exists i) S'
\end{equation}
where $i\in \Nom$ and $S'$ is the set of the sentences of the
signature that extends $\Nom$ with the nominal variable $i$. 
The models upgrade the respective concepts of Kripke structures to
$(W,M)$ by adding to $W$ interpretations of the nominals, i.e. 
$W = (|W|,\{ W_i \in |W| \mid i\in \Nom \},W_\lambda)$.
The satisfaction relations between models (i.e. Kripke structures) and
sentences extend the satisfaction relations of the corresponding
non-hybrid modal institutions with 
\begin{itemize}

\item $((W,M) \models^w \sen{i}) = (W_i = w)$; 

\item $((W,M) \models^w @_i \rho) = ((W,M) \models^{W_i} \rho)$; and 

\item $((W,M) \models^w (\exists i)\rho) = 
\bigvee \{(W',M) \models^w \rho \mid W' \mbox{ \ expansion of \ }W
\mbox{ \ to \ } \Nom\!+\!i \}$. 

\end{itemize}
Note that quantifiers over nominals allow us to simulate the binder
operator $(\downarrow \rho)$ of \cite{goranko96} by 
$(\forall i) i\impl\rho$.  

The translation of sentences and model reducts corresponding to
signature morphisms are canonical extensions of the corresponding
concepts from $\MPL$ and $\MFOL$. 

The stratifications of $\HPL$ and $\HFOL$ are like for $\MPL$ and
$\MFOL$, i.e. $\sem{(W,M)}_{(\Nom,\Sigma)} = |W|$. 
\end{rm}
\end{example}

\begin{example}[Polyadic modalities ($\MMPL$, $\MHPL$, $\MMFOL$,
  $\MHFOL$)] \label{poly-ex}
\begin{rm}
Multi-modal logics (e.g. \cite{gabbay-mml2003}) exhibit several
modalities instead of only the traditional $\Diamond$ and $\Box$ and
moreover these may have various arities.   
If one considers the sets of modalities to be variable then they have
to be considered as part of the signatures. 
We may extend each of $\MPL$, $\HPL$, $\MFOL$ and $\HFOL$ to the
multi-modal case, 
\begin{itemize}

\item by adding an `$M$' in front of each of
these names;
 
\item by adding a component $\Lambda = (\Lambda_n)_{n\in\omega}$ to
  the respective signature concept (with $\Lambda_n$ standing for the
  modalities symbols of arity $n$), e.g. an $\MHFOL$ signature would
  be a tuple of the form $(\Nom,\Lambda,(F,P))$; 

\item by replacing in the respective
  grammars the rule $S ::= \Diamond S$ by the set of rules 
\[
\{ S ::= \langle \lambda \rangle S^n \mid \lambda\in \Lambda_{n+1}, n\in\omega \};
\]
\item by replacing the binary relation $W_\lambda$ from the models
  $(W,M)$ with a set of interpretations 
  $\{ W_\lambda \subseteq |W|^n \mid \lambda\in\Lambda_n, n\in\omega \}$.   

\end{itemize}
Consequently the definition of the satisfaction relation gets upgraded
with 
\[
\mbox{ \ for each \ }
\lambda\in\Lambda_{n+1}, \  
((W,M) \models^w \langle\lambda\rangle(\rho_1,\dots,\rho_n)) =
\big(\bigvee_{(w,w_1,\dots,w_n)\in W_\lambda}\bigwedge_{1\leq i\leq n}
(W,M)\models^{w_i} \rho_i\big).
\]
The stratification is the same like in the previous examples,
i.e. $\sem{(W,M)}_{(\Nom,\Lambda,\Sigma)} =|W|$. 
\end{rm}
\end{example}

\begin{example}[Modalizations of institutions; $\HHPL$]
  \label{hhpl-ex}
\begin{rm}
In a series of works \cite{ks,HybridIns,QVHybrid} modal logic and
Kripke semantics are developed by abstracting away details that do not
belong to modality, such as sorts, functions, predicates, etc. 
This is achieved by extensions of abstract institutions (in the
standard situations meant in principle to encapsulate the atomic part
of the logics) with the essential ingredients of modal logic and
Kripke semantics. 
The result of this process, when instantiated to various concrete
logics (or to their atomic parts only) generate uniformly a wide range
of hierarchical combinations between various flavours of modal logic
and various other logics. 
Concrete examples discussed in \cite{ks,HybridIns,QVHybrid} include
various modal logics over non-conventional structures of relevance in
computing science, such as partial algebra, preordered algebra, etc. 
Various constraints on the respective Kripke models, many of them
having to do with the underlying non-modal structures, have also been
considered. 
All these arise as examples of stratified institutions like the
examples presented above in the paper. 
This great multiplicity of non-conventional modal logics constitute
an important range of applications for this work.

An interesting class of examples that has emerged quite smoothly out
of the general works on hybridization\footnote{I.e. Modalization including
  also hybrid logic features.} of institutions is that of
multi-layered hybrid logics that provide a logical base for specifying
hierarchical transition systems (see \cite{madeira-phd}).  
As a single simple example let us present here the double layered
hybridization of propositional logic, denoted $\HHPL$.\footnote{Other
  interesting examples that may be obtained by double or multiple
  hybridizations of logics would be $\HHFOL$, $\HHHPL$, etc., and also
  their polyadic multi-modalities extensions.}   
This amounts to a hybridization of $\HPL$, its models thus being
``Kripke structures of Kripke structures''.  

The $\HHPL$ signatures are triples $(\Nom^0,\Nom^1,P)$ with
$\Nom^0$ and $\Nom^1$ denoting the nominals of the first and second
layer of hybridization, respectively. 
The $(\Nom^0,\Nom^1,P)$-sentences are built over the two hybridization
layers by taking the $(\Nom^0,P)$-sentences as atoms in the grammar
for the $\HPL$ sentences with nominals from $\Nom^1$.   
In order to prevent potential ambiguities, in general we tag the
symbols of the respective layers of hybridization by the superscripts
$0$ (for the first layer) and $1$ (for the second layer).  
This convention should include nominals and connectives ($\Diamond$,
$\wedge$, etc.) as well as quantifiers. 
For instance, the expression  $@_{j^1} k^0 \wedge^1 \Box^1 \rho$ is a
sentence of $\HHPL$ where the symbols $k$ and $j$ represent nominals
of the first and second level of hybridization and $\rho$ a $\PL$
sentence. 
On the other hand, according to this tagging convention the expression
$@_{j^0} k^1 \wedge^1 \Box^1 \rho$ would not parse. 

Our tagging convention extends also to $\HHPL$ models.
A $(\Nom^0,\Nom^1,P)$-model is a pair $(W^1,M^1)$ with $W^1$ being a
$\Mod^\BRELC (\lambda)$ model and $M^1 = ((M^1)^w)_{w\in |W^1|}$ where
$(M^1)^w$ is a $(\Nom^0,P)$-model in $\HPL$, denoted $((W^0)^w,(M^0)^w)$. 
We also require that for all $w,w'\in |W^1|$, we have that
$|(W^0)^w|=|(W^0)^{w'}|$ and $(W^0)^w_i = (W^0)^{w'}_i$ 
for each $i\in \Nom^0$. 

These definitions extend in the obvious way to signature morphisms,
sentence translations, model reducts and satisfaction relation. 
We leave these details as exercise for the reader. 
Then $\HHPL$ has the same stratified structure like $\HPL$ and
$\HFOL$, namely $\sem{(W^1,M^1)}_{(\Nom^0,\Nom^1,P)} = |W^1|$.

It is easy to see that in $\HHPL$ the semantics of the Boolean
connectors and of the quantifications with nominals of the lower layer
is invariant with respect to the hybridization layer; this means that
in these cases the tagging is not necessary.
For example if $\rho$ is an $\HPL$ sentence then $(\forall^1 i^0)\rho$
and $(\forall^0 i^0)\rho$ are semantically equivalent, while if $\rho$
is not an $\HPL$ sentence (which means it has some ingredients from
the second layer of hybridization) then $(\forall^0 i^0)\rho$ would
not parse. 
In both cases just using the notation $(\forall i^0)$ would not carry
any ambiguities. 
\end{rm}  
\end{example}

The next series of examples include multi-modal first order logics
whose semantics are given by ordinary first order rather than
Kripke structures.   

\begin{example}[Multi-modal open first order logic ($\OFOL$, $\MOFOL$,
  $\HOFOL$, $\HMOFOL$)]\label{ofol-ex}
\begin{rm}
The stratified institution $\OFOL$ is the $\FOL$ instance of $St(\I)$, the
`internal stratification' abstract example developed in \cite{strat}. 
An $\OFOL$ signature is a pair $(\Sigma,X)$ consisting of $\FOL$
signature $\Sigma$ and a finite block of variables. 
An $\OFOL$ signature morphism $\varphi \co (\Sigma,X) \ra
(\Sigma',X')$ is just a $\FOL$ signature morphism $\varphi \co \Sigma
\ra \Sigma'$ such that $X \subseteq X'$. 

We let $\Sen^\OFOL((F,P),X)=\Sen^\FOL(F+X,P)$ and 
$\Mod^\OFOL ((F,P),X) = \Mod^\FOL(F,P)$. 

For each $((F,P),X)$-model $M$, each $w\in |M|^X$, and each
$((F,P),X)$-sentence $\rho$ we define 
\[
(M (\models^\OFOL_{(F,P),X})^w \rho) = (M^w \models^\FOL_{(F+X,P)} \rho)
\]
where $M^w$ is the expansion of $M$ to $(F\!+\!X,P)$ such that 
$M^w_X = w$. 
This is a stratified institution with $\sem{M}_{\Sigma,X} = |M|^X$
for each $(\Sigma,X)$-model $M$.
For any signature morphism $\varphi \co (\Sigma,X) \ra (\Sigma',X')$
and any $(\Sigma',X')$-model $M'$, $\sem{M'}_\varphi \co |M'|^{X'} \ra
|M'|^X$ is defined by $\sem{M'}_\varphi (a) = a|_X$ (i.e. the
restriction of $a$ to $X$). 
Note that $\sem{M'}_\varphi$ is surjective and that this provides an
example when $\sem{\_}$ is a proper lax natural transformation. 

We may refine $\OFOL$ to a multi-modal logic ($\MOFOL$) by adding 
\[
\{ S ::= \langle\pi\rangle S^n \mid \pi\in P_{n+1}, n\in\omega \}
\]
to the grammar defining each $\Sen^\OFOL ((F,P),X)$ and consequently  
by extending the definition of the satisfaction relation with 
\begin{itemize}

\item 
$(M \models^w \langle\pi\rangle (\rho_1,\dots,\rho_n)) =
\bigvee_{(w,w_1,\dots,w_n)\in (M^X)_\pi} \bigwedge_{1\leq i\leq
  n}(M\models^{w_i}\rho_i)$ for each $\pi\in P_{n+1}$, $n\in\omega$.

\end{itemize}
(Here and elsewhere $M^X$ denotes the $X$-power of $M$ in the category
of $\FOL$ $(F,P)$-models.)

Or else we may refine $\OFOL$ with nominals ($\HOFOL$) by adding the
grammar for nominals (\ref{hybrid-grammar}), for each constant $i\in
F_0$, to the grammar defining each $\Sen^\OFOL ((F,P),X)$ and
consequently extending the definition of the satisfaction relation
with   
\begin{itemize}

\item $(M\models_{(F,P),X}^w \sen{i}) = ((M^X)_i = w)$;

\item $M \models_{(F,P),X}^w @_i \rho) = 
(M \models_{(F,P),X}^{(M^X)_i} \rho)$;

\item $(M\models_{(F,P),X}^w (\exists i)\rho) = 
\bigvee \{M' \models_{(F\!+\!i,P),X}^w \rho \mid M' \mbox{ \ expansion of \ }M
\mbox{ \ to \ } (F\!+\!i,P) \}$.

\end{itemize}
We can also have $\HMOFOL$ as the blend between $\HOFOL$ and
$\MOFOL$. 
\end{rm}
\end{example}

\subsection{Stratified institution morphisms}\label{si-morphism}

They extend the concept of institution morphism (Definition
\ref{institution-morphism}) from ordinary institutions to stratified
institutions. 
The 2-dimensional aspect of the stratified institutions leads to a
higher complexity in the following definition of morphisms of
stratified institutions. 
This concept will be instrumental when defining our concept of
decompositions of stratified institutions. 

\begin{definition}[Morphism of stratified institutions]
  \label{strat-institution-morphism}
Given two stratified $\SI$ and $\SI'$ a \emph{stratified institution
  morphism}  $(\Phi, \alpha, \beta) \co \SI' \ra \SI$ consists of 
\begin{itemize}

\item a functor $\Phi \co \Sign' \to \Sign$,

\item a natural transformation $\alpha \co \Phi;\Sen \Ra \Sen'$,
  and 

\item a lax natural transformation
  $\beta\co \Mod' \Ra \Phi^{\op} ; \Mod$ such that
  $\beta;\Phi^{\op} \sem{\_} = \sem{\_}'$, 

\end{itemize}
and such that the following Satisfaction Condition holds for  
any $\SI'$-signature $\Sigma'$, any $\Sigma'$-model $M'$, any $w \in
\sem{M'}_{\Sigma'}$ and any 
$\Phi(\Sigma')$-sentence $\rho$:
\[
M' \models^{w'} \alpha_{\Sigma'} \rho \mbox{ \ if and only if \ }
\beta_{\Sigma'} M' \models^{w'}  \rho.
\]
When $\beta$ is strict, $(\Phi,\alpha,\beta)$ is called \emph{strict} too. 
\end{definition}
The condition on $\beta$ means the following:
\begin{itemize}

\item 
for each
  $\SI'$-signature $\Sigma$ the following diagram commutes
  \[
    \xymatrix{
      \Mod'(\Sigma) \ar[r]^{\beta_\Sigma} \ar[d]_{\sem{\_}'_{\Sigma}}
        & \Mod(\Phi\Sigma) \ar[dl]^{\sem{\_}_{\Phi\Sigma}}\\
        \Set &
    }
  \]

\item
  for each $\SI'$-signature morphism $\varphi \co \Sigma \to \Omega$
  \[
  \beta_\Omega \sem{\_}_{\Phi\varphi} ; \beta_\varphi
  \sem{\_}_{\Phi\Sigma} = \sem{\_}'_\varphi 
  \]
  which can be visualised as the commutativity of the following
  diagram:
  \[
    \xymatrix @R+1em {
      &
      \Mod'(\Omega) \ar[r]^{\beta_\Omega} \ar[d]|{\Mod'(\varphi)}
      \ar@/_{2pc}/[dd] \ar@/_{5pc}/[dd]_{\sem{\_}'_\Omega}
  \ar@/^.9pc/[dr]_{\quad} 
  \ar@/_.7pc/[dr] 
      & \Mod(\Phi\Omega) \ar[d]|{\Mod(\Phi\varphi)}
        \ar@{}[ld]^(.35){}="c"^(.62){}="d" 
        \ar@{=>} "c";"d"_{ \ \ \beta_{\varphi}}  
      \ar@/^{2pc}/[dd] \ar@/^{5pc}/[dd]^{\sem{\_}_{\Phi\Omega}} & \\
      \ar@{}[r]^(-.15){}="c"^(.52){}="d" 
      \ar@{=>} "c";"d"^{ \ \ \sem{\_}'_{\varphi}}
      &
      \Mod'(\Sigma) \ar[r]^{\beta_\Sigma} \ar[d]|{\sem{\_}'_{\Sigma}}
      & \Mod(\Phi\Sigma) \ar[d]|{\sem{\_}_{\Phi\Sigma}} &
      \ar@{}[l]^(-.15){}="c"^(.52){}="d" 
      \ar@{=>} "c";"d"_{ \ \ \sem{\_}_{\Phi\varphi}}\\
      &
      \Set \ar@{-}[r]_{=} & \Set & 
    }
  \]
  
\end{itemize}

Morphisms of stratified institutions form a category under a
composition that is defined component-wise like in the case of
morphisms of ordinary institutions:
\[
  (\Phi',\alpha',\beta') \ ; \ (\Phi,\alpha,\beta) =
  (\Phi;\Phi \ , \ \alpha\Phi';\alpha' \ , \ \beta';\beta\Phi^{\opp}).
\]

\begin{fact}\label{strict-fact}
For any morphism of stratified institutions $\SI' \to \SI$, if $\SI$
is strict then $\SI'$ is strict too.
\end{fact}

The proof of the following result consists of straightforward
verifications; we will skip them.

\begin{proposition}
Any strict stratified institution morphism $(\Phi,\alpha,\beta) \co
\SI' \to \SI$ induces an institution morphism
$(\Phi,\alpha,\beta^\sharp)\co \SI'^\sharp \to \SI^\sharp$ where for
each $\SI'$-signature $\Sigma$, each $\Sigma$-homomorphism $h \co M
\to N$ and each $w\in \sem{M}'_\Sigma$
\[
\beta^\sharp_\Sigma h = \beta_\Sigma h \co (\beta_\Sigma M, w) \to
(\beta_\Sigma N, \sem{h}'_\Sigma w). 
\]
\end{proposition}

\subsection{Nominals in stratified institutions}

The definitions of this section are inherited from
\cite{KripkeStrat}. 

\begin{definition}[Nominals extraction]
  \label{nom-extraction}
Given a stratified institution $\SI$, a \emph{nominals extraction} is a
pair $(N,\Nm)$ consisting of a functor $N \co \Sign^\SI \ra
\Sign^\SETC$ and a lax natural transformation $\Nm \co \Mod^\SI \Ra
N;\Mod^\SETC$ such that $\sem{\_} = \Nm ; N(\Mod^\SETC \Ra \SET)$. 
\end{definition}

\begin{fact}
A nominals extraction $(N,\Nm)$ is precisely a stratified institution 
morphism
\[
  (N,\emptyset,\Nm) \co \SI \to \SETC
\]
where $\SETC$ is considered as a stratified institution with no
sentences and for each $\SETC$-model $M$, $\sem{M} = |M|$ (the
underlying set of $M$).  
\end{fact}

\begin{example}\label{nom-extraction}
\begin{rm}
The following table shows some nominals extractions for the stratified
institutions introduced above.  
Note that $\HHPL$ admits two such nominals extractions. 

\

\noindent
\begin{scriptsize}
\begin{tabular}{rll}
stratified institution & $N$ & $\Nm$ \\\hline
$\HPL, \HFOL, \MHPL, \MHFOL$ & $N(\Nom,\Sigma) = \Nom$ &
$\Nm_{(\Nom,\Sigma)} (W,M) = (|W|,(W_i)_{i\in \Nom})$ \\\hline
$\HHPL$ & $N(\Nom^0,\Nom^1,P) = \Nom^0$ & $\Nm (W^1,M^1) =
(|(W^0)^w|,((W^0)^w_i)_{i\in \Nom^0})$ \\ 
        & $N(\Nom^0,\Nom^1,P) = \Nom^1$ & $\Nm (W^1,M^1) =
        (|W^1|,(W^1_i)_{i\in\Nom^1})$ \\\hline 
$\HOFOL, \HMOFOL$ & $N((F,P),X) = F_0$ & 
         $\Nm (M) = (|M|^X, ((M^X)_i)_{i\in F_0})$ \\
\end{tabular}
\end{scriptsize}
\end{rm}
\end{example}

\begin{definition}\label{hybrid-strat-institution-dfn}
Let $\SI$ be a stratified institution endowed with a nominals
extraction $N, \Nm$. 
For any $i\in N(\Sigma)$
\begin{itemize}

\item a $\Sigma$-sentence $\sen{i}$ is an \emph{$i$-sentence} when 
\[
(M \models^w \sen{i}) = ((\Nm_\Sigma (M))_i = w);
\]

\item for any $\Sigma$-sentence $\rho$, a $\Sigma$-sentence $@_i
  \rho$ is the \emph{satisfaction of $\rho$ at $i$} when 
\[
(M \models^w @_i \rho) = (M \models^{(\Nm_\Sigma (M))_i} \rho)
\]
for each $\Sigma$-model $M$ and for each $w\in \sem{M}_\Sigma$.
\end{itemize}
The stratified institution $\SI$ has \emph{explicit local
  satisfaction} when there exists a satisfaction at $i$ for each
sentence and each appropriate $i$. 
\end{definition}

\begin{example}\label{conn-table}
\begin{rm}
The following table shows what of the properties of
Definition \ref{hybrid-strat-institution-dfn} are satisfied by the
examples of stratified institutions given above in the paper.  

\

\noindent
\begin{small}
\begin{tabular}{rcc}
 & $\sen{i}$ & $@_i$ \\\hline 
$\MPL / \MFOL / \MMPL / \MMFOL / \OFOL / \MOFOL$ & & \\\hline
$\HPL / \HFOL / \MHPL / \MHFOL / \HOFOL / \HMOFOL$ & \checkmark &
      \checkmark \\\hline
$\HHPL$  & $\sen{i^0}$, $\sen{i^1}$ & $@_{i^0}$, $@_{i^1}$ 
\end{tabular}
\end{small}
\end{rm}
\end{example}

\section{Decompositions of stratified institutions}
\label{decomp-sec}

An analysis of the structure of conventional Kripke semantics reveals
the following situation for an individual Kripke model:
\begin{itemize}

\item There is a \emph{family} of models in a ``lower'' logical system,
  usually propositional or first order logic.
  The indexing of the family is what is usually refereed to as
  ``worlds''. 

\item Then there is a certain structure imposed upon this family of
  models.
  This happens at the level of the ``worlds'' commonly in the form of 
  relations. 
  
\end{itemize}
In this section we address this general structure of Kripke semantics
from an abstract axiomatic perspective. 
The result is a general abstract class of stratified institutions that
does not necessarily consider explicitly Kripke frames nor modalised
sentences, but which retains in an abstract form the essential idea of
a stratified institution in which a ``header'' institution structures
a certain multiplication of a ``base'' institution.

The contents of this section is as follows:
\begin{enumerate}

\item We introduce the concept of ``base'' of a stratified institution
  $\SI$ which represents the institution in which the ``worlds'' are
  interpreted as models.

\item We extend the reduction of a stratified institution $\SI$ to an
  ordinary institution $\SI^\sharp$ introduced in Section
  \ref{red-strat-sec} to an adjunction between the categories of
  stratified institutions and of ordinary institutions.

\item Finally, on the basis of the above mentioned adjunction we
  introduce the main concept of this section, namely that of
  decomposition of a stratified institution. 
  
\end{enumerate}

\subsection{Bases for stratified institutions}

In a stratified institution $\SI$ with Kripke semantics, if $M$ is a
Kripke model and $w \in \sem{M}$ then the model $(M,w)$ of
$\SI^\sharp$ represents a ``localisation'' in $M$ of the ``world''
$w$.
This corresponds to a model in ``lower'' institution. 
However the construction of $\SI^\sharp$ is independent of the fact
that $M$ is really a Kripke model, so this process of semantic
localisation is a very general one.
On the other hand we should be able to have the (syntax of the)
``lower'' logic available at the level of $\SI$.
These ideas are captured by the following definition.

\begin{definition}[Base of stratified institution]
A \emph{base for a stratified institution $\SI$} is an institution
morphism $(\Phi,\alpha,\beta) \co \SI^\sharp \to B$.
A base is \emph{shared} when for each signature $\Sigma$, each
$\Sigma$-model $M$ of $\SI$, and any $w,w' \in
\sem{M}_{\Sigma}$  we have that $\beta_\Sigma (M,w) =
\beta_\Sigma (M,w')$. 
\end{definition}

\begin{example}[Base for $\MPL$]
  \label{base0}
\begin{rm}
For the stratified institutions that are based on some form of Kripke
semantics we may consider $\B$ to be the institution that at the
syntactic level removes from $\SI$ all syntactic entities that
involve modalities, and whose models are the individual ``worlds'' of
the respective Kripke semantics. 
For instance, in the case of $\MPL$:
\begin{itemize}

\item $\B = \PL$ and $\Phi$ is the identity functor on $\Set$, 

\item $\alpha_P$ is the inclusion $\Sen^{\PL} (P) \subseteq
  \Sen^{\MPL} (P)$, 

\item $\beta_P (M,w) = M^w$, etc. 

\end{itemize}
\end{rm}
\end{example}

\begin{example}[Shared base for $\MFOL$]
  \label{base1}
\begin{rm}
For the stratified institutions with Kripke semantics based on first
order models with some form of sharing, $\B$ may remove even more
structure from $\SI$ such that at the semantic level $\beta$ maps a
Kripke structure to the respective shared underlying domain.
For instance, if $\SI$ is $\MFOL$:
\begin{itemize}

\item $\B$ is $\SETC$, i.e. the sub-institution of $\FOL$ induced by
  signatures that contain only constants, $\Phi$ removes from the 
  signatures the predicates and the non-constant operations, 

\item $\alpha$ is empty (as $\SETC$ does not contain any sentences),
  and 

\item $\beta_{(F,P)} (M,w) = (|M|, (M_c)_{c\in F_0})$ where $|M|$ is the
  shared underlying domain of $M$,  and each $M_c$ is the shared
  interpretation of the constant $c$.  

\end{itemize} This has been an example of a shared base.
\end{rm}
\end{example}

\begin{example}[Non-shared base for $\MFOL$]
When we allow more structure for $\B$ we obtain another relevant base
for $\MFOL$.
\begin{itemize}

\item We let $\B = \FOL$; then $\Phi$ is identity.

\item $\alpha$ consists of the canonical inclusions of sets of
  sentences. 

\item $\beta_{(F,P)} (M,w) = M^w$. 
  
\end{itemize}
\end{example}

All examples of Section \ref{ex-sect} admit various bases in the
manner of the previous couple of examples.
For instance $\HHPL$ may admit $\HPL$ as a base.


\subsection{The adjunction between stratified and ordinary
  institutions}

The representation of stratified institutions as ordinary institutions
given by the flattening of Fact \ref{sharp-institution-fact} is part
of an adjunction which is instrumental for defining the main concept
introduced in this work, that of decompositions of stratified
institutions. 

Let $\INS$ be the category of institution morphisms and $\SINS$ be the
category of strict stratified institution morphisms.

\begin{proposition}
Let $(\_)^\sharp \co \SINS \to \INS$ be the canonical extension of the
mapping $\SI \mapsto \SI^\sharp$ defined in Fact
\ref{sharp-institution-fact}.
Then $(\_)^\sharp$ has a right adjoint which we denoted as $\wt{(\_)}
\co \INS\to \SINS$. 
\end{proposition}

\begin{proof}
For any institution $\B$ we define the following stratified
institution $\wt{\B}$:
  \begin{itemize}\itemsep=1pt

  \item $\Sign^{\widetilde{\B}} = \Sign^{\B}$ and
    $\Sen^{\widetilde{\B}} = \Sen^{\B}$, 

  \item $|\Mod^{\widetilde{\B}} (\Sigma)| = \{ (W, B \co W \to
    |\Mod^{\B}(\Sigma)|) \mid W \text{ set} \}$,

  \item $\Mod^{\widetilde{\B}} (\Sigma) ((W,B), (V,N))$ consists of $h
    = (h_0 \co W \to V, (h^w \co B^w \to N^{h_0 w})_{w\in W})$,

  \item for each signature morphism $\varphi \co \Sigma \to \Sigma'$
    and each $\Sigma'$-model $(W',B')$:\\
    $\Mod^{\widetilde{\B}} (\varphi)(W',B') = (W',B';\Mod^\B (\varphi))$,

  \item $\sem{W,B}^{\widetilde{\B}}_\Sigma = W$ and $\sem{h}_\Sigma =
    h_0$,

  \item $\sem{\_}_\varphi$ are identities, and 

  \item $(W,B) (\models^{\widetilde{\B}}_\Sigma)^w \rho$ if and only
    if $B^w \models^\B_\Sigma \rho$. 
    
  \end{itemize}
  The proof that $\wt{\B}$ is a stratified institution consists of
  straightforward verifications. 
  Let us do only the Satisfaction Condition:
  \[
    \begin{array}{rll}
      (W',B') \models^w \alpha \rho & \text{iff }
            B'^w \models \alpha \rho & \text{by definition} \\
      & \text{iff } \Mod^\B (\varphi)B'^w \models \rho &
        \text{by the Satisfaction Condition in } \B \\
      & \text{iff }(W',B';\Mod(\varphi)) \models^w \rho &                             \text{by definition}\\                 
      & \text{iff } \Mod^{\widetilde{\B}} (\varphi) (W',B')
        \models^w \rho & \text{by definition.}
    \end{array}
  \]

  Then $\wt{(\_)}$ extends canonically to a functor $\INS\to\SINS$.
  In order to prove that this is a right adjoint to $(\_)^\sharp$ we
  first define the co-unit of the adjunction as follows.
  For each institution $\B$ we let the institution morphism
  $\varepsilon_{\B} \co \wt{\B}^\sharp \to \B$ have identities for the
  signature and sentence translation functors and maps each
  $\wt{\B}^\sharp$ $\Sigma$-model $((W,(B^w)_{w\in W}),w)$ to $B^w$.
  Then we prove the universal property of $\wt{\B}$, namely that for
  each base $(\Phi,\alpha,\beta) \co \SI^\sharp \to \B$ there exists
  an unique strict stratified institution morphism
  $(\Phi,\alpha,\wt{\beta}) \co \SI \to \wt{\B}$ such that the
  following diagram commutes:
  \begin{equation}\label{adjoint-diag}
    \xymatrix{
      \B & & \wt{\B}^\sharp \ar[ll]_{\varepsilon_{\B}} & & & \wt{\B} \\
      & \SI^\sharp \ar[ul]^{(\Phi,\alpha,\beta)}
      \ar[ur]_{(\Phi,\alpha,\wt{\beta})^\sharp} & & & \SI
      \ar[ur]_{(\Phi,\alpha,\wt{\beta})} 
    }
  \end{equation}
  Because the signature and the sentences translation functors of
  $\varepsilon_{\B}$ are identities there is no other choice for the
  signature and the sentence translation functors of
  $(\Phi,\alpha,\wt{\beta})$.
  By the commutativity \eqref{adjoint-diag} it follows that the
  definition of $\wt{\beta}$ is constrained to
  \[
  \widetilde{\beta}_\Sigma M = (\sem{M}_\Sigma (\beta_\Sigma
  (M,w))_{w\in\sem{M}_\Sigma}). 
  \]
  We may skip a few straightforward things related to establishing
  that $(\Phi,\alpha,\wt{\beta})$ is indeed a strict stratified
  institution morphism and only show its Satisfaction Condition:
  \[
    \begin{array}{rll}
      \widetilde{\beta}M \models^w \rho & \text{iff }
            \beta (M,w) \models \rho & \text{by definition} \\
      & \text{iff } (M,w) \models \alpha \rho &
        \text{by the Satisfaction Condition in } \SI^\sharp \\
      & \text{iff }M \models^w \alpha \rho & \text{by definition}.
    \end{array}
  \]  
\end{proof}

As a matter of notation, in what follows, for any base
$(\Phi,\alpha,\beta) \co \SI^\sharp \to \B$ we will denote its
correspondent through the natural isomorphism
$\INS(\SI^\sharp, \B) \cong \SINS(\SI,\wt{\B})$ by
$(\Phi,\alpha,\wt{\beta})$. 

\subsection{Decompositions of stratified institutions}

In many Kripke semantics examples the models are subject to certain
constraints, such as for instance the sharing constraints discussed in
Example \ref{mfol-ex} ($\MFOL$) or in Example \ref{hhpl-ex} ($\HHPL$).
Such constraints are treated abstractly in the following definition as
a sub-functor of the model functor of $\wt{\B}$. 

\begin{definition}[Decomposition of stratified institution]
  \label{decomp-def}
  Let $\SI$ be a stratified institution and
  $(\Phi,\alpha,\beta) \co S^\sharp \to \B$ be a base for $\SI$.
  Let $\Mod^C \subseteq \Mod^{\widetilde{\B}}$ be a sub-functor such
  that for each signature $\Sigma$,
  \[
    \wt{\beta}_\Sigma (\Mod^{\SI} (\Sigma)) \subseteq \Mod^C
    (\Phi\Sigma),
  \]
  refereed to as the \emph{constraint model sub-functor}. 
  Let $\widetilde{\B}^C$ denote the stratified sub-institution of
  $\widetilde{\B}$ induced by $\Mod^C$. 
  A \emph{decomposition of $\SI$} consists of two stratified
  institution morphisms like below
\[
  \xymatrix @C+3em {
  \SI^0 & \ar[l]_{(\Phi^0,\alpha^0,\beta^0)} \SI
  \ar[r]^{(\Phi,\alpha,\widetilde{\beta})}& \widetilde{\B}^C  
  }
\]
such that for each $\SI$-signature $\Sigma$
\[
  \xymatrix @C+2em {
  \Mod^0 (\Phi^0 \Sigma) \ar[dr]_{\sem{\_}^0_{\Phi^0 \Sigma}} &
  \Mod^\SI (\Sigma) \ar[l]_{\beta^0_\Sigma} 
  \ar[r]^{\widetilde{\beta}_\Sigma}
  \ar[d]^{\sem{\_}^\SI_\Sigma} & \Mod^C (\Phi \Sigma)
  \ar[dl]^{\sem{\_}^{\widetilde{\B}}_{\Phi\Sigma}} \\ 
   & \Set & 
  }
\]  
is a pullback in $\CAT$.
\end{definition}
The following important property follows from Fact \ref{strict-fact}
because $\widetilde{\B}$ is strict and so is $\widetilde{\B}^C$. 

\begin{fact}\label{strict-decomp}
Any stratified institution that admits a decomposition is strict.
\end{fact}

Let us note the following aspects emerging from Definition
\ref{decomp-def}.
\begin{itemize}

\item The models of $\SI$ can be represented as pairs of
  $\SI^0$-models and families of $\B$-models satisfying certain
  constraints (hence $\wt{\B}^C$ models) such that the ``worlds'' of
  the corresponding $\wt{\B}^C$ model constitutes the stratification
  of the corresponding $\SI^0$ model.
  This means that at the semantic level $\SI$ is completely determined
  by the two components of the decomposition. 

\item The situation at the syntactic level is different.
  The syntax (signatures and sentences) of each of the two components
  is represented in the syntax of $\SI$, but the latter is not
  completely determined by the former syntaxes.
  In other words $\SI$ may have signatures and sentences that do not
  originate from either of the two components.
  This is what Definition \ref{decomp-def} gives us.
  However while there are hardly any examples  / applications where
  any sentence comes from one of the two components, in many examples
  the signatures of $\SI$ are composed from the signatures of $\SI^0$
  and those from $\B$.
  In Lemma \ref{prod-signature} below we will provide a general such
  situation. 
  
\end{itemize}

\begin{example}[Decompositions of $\MPL, \HPL$]\label{dec1}
\begin{rm}
The sub-institutions $\REL$, $\BREL$, $\RELC$, $\BRELC$, $\SETC$ of
$\FOL$ can be regarded trivially as stratified institutions by letting
for each model $W$ and sentence $\rho$, $\sem{W} = |W|$ (i.e. the
underlying set of $W$) and $(W \models^w \rho) = (W \models \rho)$ for
each $w \in |W|$.
Under this perpective we let
\[
  \SI^0 =
  \begin{cases}
    \BREL, & \SI = \MPL, \\
    \BRELC, & \SI = \HPL. 
    \end{cases} 
  \]
  Then
  \begin{itemize}

  \item The functor $\Phi^0$ forgets / erases the sets $P$ of
    predicate symbols from the signature.

  \item $\alpha^0_\Sigma$ are empty in the case of $\MPL$ (because
    $\BREL$ does not have sentences) and is defined by
    \[
      \alpha^0_{(\Nom,P)} \lambda(i,j) = @_i \Diamond j
      (= @_i \neg \Box \neg j).
    \]
    for the atoms and then for any sentence by induction  on the
    structure of the respective sentence such that
    $\alpha^0_{(\Nom,P)}$ commutes with the connectives (Boolean and
    quantifiers).

  \item $\beta^0_{(\Nom,P)} (W,M) = W$.

  \item The Satisfaction Condition of $(\Phi^0, \alpha^0, \beta^0)$
    can be verified easily by induction on the structure of the
    sentences.  
    
  \end{itemize}

  The bases of $\SI$ are those of Example \ref{base0}, i.e. $\B =
  \PL$, etc.
  In both situations the model constraint model sub-functor $\Mod^C$
  is just $\Mod^{\widetilde{\B}}$, so it is not a proper sub-functor.
  With respect to the $R$, $S4$, \etc \ variants of $\MPL$ and $\HPL$, 
  in these cases $\Mod^0 (\Sigma)$ is restricted to those
  (sub-)categories of relations that satisfy the respective
  constraints. 
\end{rm}
\end{example}

\begin{example}[Decompositions of $\MFOL, \HFOL$]\label{dec2}
\begin{rm}
  The decompositions of $\MFOL$ and $\HFOL$ parallel the
  decompositions of $\MPL$ and $\HPL$, respectively, by
  replacing $\PL$ with $\FOL$ in Example \ref{dec1}.
  The only significant difference that needs a special mention is at
  the level of the constraint model sub-functor $\Mod^C$.
  In both $\MFOL$ and $\HFOL$, a model $(W,B)$ belongs to $\Mod^C
  (F,P)$ when $| B^w | = | B^v |$ and $B^w_c = B^v_c$ for all $w,v \in
  |W|$ and all $c \in F_0$.
  Likewise a model homomorphism $(h^0, (h^w)_{w\in |W|}) \co (W,B) \to
  (V,N)$ is such that $h^w = h^v$ for all $w,v \in |W|$. 
\end{rm}
\end{example}

\begin{example}[Decompositions of polyadic modalities stratified
  institutions]\label{dec3}
\begin{rm}
The cases of $\MMPL, \MMFOL, \MHPL, \MHFOL$ are similar to those of
Examples \ref{dec1} and \ref{dec2} by taking $\SI^0 = \REL$ when $\SI$
is non-hybrid and $\SI^0 = \RELC$ when $\SI$ is hybrid. 



  
\end{rm}
\end{example}

Apart of the stratified institutions of Example \ref{ofol-ex} which
are not strict all other stratified institutions given as examples in
Section \ref{ex-sect} admit decompositions in a similar manner as
in the examples of this section.
But Definition \ref{decomp-def} has a theoretical potential related to
the $\SI^0$ component that may generate situations much beyond Kripke
semantics in its common acceptations.
For instance we may consider $\SI^0$ to be an institution of algebras,
which will mean algebraic operations on the ``worlds'' in $\SI$
models.
To unleash the full potential of Definition \ref{decomp-def} in this
direction is an interesting topic of further investigation.

\subsubsection*{Implicit nominal structures via decompositions}
In the applications the eventual nominal structures of $\SI$ come from
$\SI^0$. 
The following fact clarifies mathematically this situation in a full
generality. 

\begin{fact}\label{extraction-import}
  Consider a decomposition of a stratified institution
  \[
  \xymatrix @C+3em {
  \SI^0 & \ar[l]_{(\Phi^0,\alpha^0,\beta^0)} \SI
  \ar[r]^{(\Phi,\alpha,\widetilde{\beta})}& \widetilde{\B}^C.  
  }
\]
Then any nominals extraction of $\SI^0$ induces canonically a
nominals extraction of $\SI$ by composition with $(\Phi^0,
\alpha^0, \beta^0)$. 
\end{fact}
In \cite{KripkeStrat} an implicit axiomatic approach to modalities is
introduced in a manner similar to Definitions \ref{nom-extraction} and
\ref{hybrid-strat-institution-dfn} and that is based on the concept of
``frame extraction''.
It is then possible to have a replica of Fact \ref{extraction-import}
for frame extractions. 

\section{Model amalgamation in stratified institutions}

In this section we study model amalgamation in the context of
stratified institutions.
We do this in two parts as follows:
\begin{enumerate}

\item We define a concept of model amalgamation specific to stratified
  institutions.

\item We develop a general result that builds the model amalgamation
  property in a stratified institution that admits a decomposition,
  from the model amalgamation properties of the two components. 

\end{enumerate}

\subsection{Concepts of model amalgamation in stratified institutions}

The following definition extends the concept of model amalgamation
\cite{sannella-tarlecki88,tarlecki86,jm-granada89,modalg,iimt,AlgStrucSpec,sannella-tarlecki-book},
etc., from ordinary institution theory to stratified institutions.
This introduces two concepts.
The first one represents just the ordinary institution theoretic
concept of model amalgamation formulated for stratified institutions
(it does not involve the stratification structure).
The second one is specific to stratified institutions. 

\begin{definition}[Model amalgamation]
  \label{stratified-amalgamation}
Consider a stratified institution $\SI$ and a commutative square of
signature morphisms like below:
\begin{equation}\label{signature-square}
  \xymatrix{
  \Sigma \ar[r]^{\varphi_1} \ar[d]_{\varphi_2} & \Sigma_1
  \ar[d]^{\theta_1} \\
  \Sigma_2 \ar[r]_{\theta_2} & \Sigma'
  }
\end{equation}
Then this square
\begin{itemize}

\item 
is a \emph{model amalgamation square} when for each $\Sigma_k$-model
$M_k$, $k= \ol{1,2}$ such that $\varphi_1 M_1 = \varphi_2 M_2$
there exists an unique $\Sigma'$-model $M'$ such that $\theta_k  M' =
M_k$, $k= \ol{1,2}$, and 

\item is a \emph{stratified model amalgamation square} when for each
  $\Sigma_k$-model $M_k$ and each $w_k \in \sem{M_k}_{\Sigma_k}$ , $k=
  \ol{1,2}$, such that $\varphi_1 (M_1) = \varphi_2 (M_2)$ and
  $\sem{M_1}_{\varphi_1} w_1 = \sem{M_2}_{\varphi_2} w_2$ there exists
  an unique $\Sigma'$-model $M'$ and an unique $w' \in
  \sem{N'}_{\Sigma'}$ such that $\theta_k M' =  M_k$ and
  $\sem{M'}_{\theta_k} w' = w_k$, $k= \ol{1,2}$. 

\end{itemize}
The model $M'$ is called the \emph{(stratified) amalgamation of $M_1$
  and $M_2$}.

When all pushout squares of signature morphisms are (stratified) model
amalgamation squares we say that $\SI$ is  \emph{(stratified)
  semi-exact}. 
\end{definition}

Definition \ref{stratified-amalgamation} can be extended to other
variants of model amalgamation in the literature.

The following straightforward fact reduces stratified model
amalgamation to ordinary model amalgamation. 

\begin{fact}
A commutative square of signature morphisms like
\eqref{signature-square} is a stratified model amalgamation square in 
$\SI$ if and only if it is a model amalgamation square in
$\SI^\sharp$.  
\end{fact}

The following result provides a couple of conditions that are
sufficient for stratified model amalgamation.
However they fall short of being also necessary conditions.
In particular this situation shows that plain model amalgamation
cannot be derived  from the seemingly more refined concept of
stratified model amalgamation. 

\begin{proposition}\label{amg-prop}
A commutative square of signature morphisms like
\eqref{signature-square} is a stratified model amalgamation square if 
\begin{itemize}

\item 
  \[
    \xymatrix @C+1em {
    \Mod(\Sigma) & \Mod(\Sigma_1) \ar[l]_{\Mod(\varphi_1)} \\
    \Mod(\Sigma_2) \ar[u]^{\Mod(\varphi_2)} & \Mod(\Sigma')
    \ar[l]^{\Mod(\theta_2)} \ar[u]_{\Mod(\theta_1)}
    }
  \]
  is a pullback in $|\CAT|$, and 

\item for each $\Sigma'$-model $M'$
  \[
    \xymatrix @C+2em {
    \sem{\varphi(\theta M')}_\Sigma & \sem{\theta_1 M'}_{\Sigma_1}
    \ar[l]_{\sem{\theta_1 M'}_{\varphi_1}} \\ 
     \sem{\theta_2 M'}_{\Sigma_2} \ar[u]^{\sem{\theta_2
         M'}_{\varphi_2}} & \sem{M'}_{\Sigma'} 
    \ar[l]^{\sem{M'}_{\theta_2}} \ar[u]_{\sem{M'}_{\theta_1}} 
    }
  \]
  is a pullback in $\Set$. 
\end{itemize}
\end{proposition}

\begin{proof}
Note that the first condition just says that \eqref{signature-square}
is a model amalgamation square.
We consider $M_1$, $w_1$, $M_2$ and $w_2$ like in the definition of
stratified model amalgamation.
Then we consider $M'$ to be the unique amalgamation of $M_1$ and $M_2$
and apply the second condition for $w_1$ and $w_2$. 
\end{proof}

Note that stratified model amalgamation implies the second condition
of Proposition \ref{amg-prop} (by considering $M_k = \theta_k (M')$)
but it does not technically imply the first condition. 

\begin{example}
\begin{rm}
When the stratification is strict then the concept of stratified model 
amalgamation collapses to that of (ordinary) model amalgamation.
For instance this is the case in $\MPL$, $\MFOL$, etc., where model
amalgamation can be thus established from the model amalgamation in
$\PL$, $\FOL$, etc. (see \cite{iimt}). 
\end{rm}
\end{example}

\begin{example}
\begin{rm}
In $\OFOL, \MOFOL, \HOFOL, \HMOFOL$ the stratification is a proper lax
natural transformation.
In all these examples ordinary model amalgamation and stratified model
amalgamation are different concepts.
Let us look in some detail into the $\OFOL$ case.
Let us consider a pushout square of $\FOL$ signature morphisms
\begin{equation}\label{amg-square2}
\xymatrix{
  \Sigma \ar[r]^{\varphi_1} \ar[d]_{\varphi_2} & \Sigma_1
  \ar[d]^{\theta_1} \\
  \Sigma_2 \ar[r]_{\theta_2} & \Sigma'
}
\end{equation}
and sets of variables $X, X_1, X_2, X'$ such that $X = X_1 \cap X_2$
and $X' = X_1 \cup X_2$.
Then
\[
\xymatrix{
  (\Sigma,X) \ar[r]^{\varphi_1} \ar[d]_{\varphi_2} & (\Sigma_1, X_1)
  \ar[d]^{\theta_1} \\
  (\Sigma_2, X_2) \ar[r]_{\theta_2} & (\Sigma', X')
}
\]
is a stratified model amalgamation square in $\OFOL$ because
\begin{itemize}

\item
it is an ordinary model amalgamation square since \eqref{amg-square2}
is a model amalgamation square in $\FOL$ as $\FOL$ is semi-exact
(according to the literature, eg. \cite{iimt}), and

\item for each $(\Sigma',X')$-model $M'$ (aka $\FOL$ $\Sigma'$-model)
  and each $a_k \co X_k \to |M_k|$, $k = \ol{1,2}$, such that $a_1 (x)
  = a_2 (x)$ for each $x\in X$, $a' \co X' \to |M'|$ defined by $a'(x)
  = a_k (x)$ when $x\in X_k$ is unique such that $\sem{M'}_{\theta_k}
  a = a_k$, $k = \ol{1,2}$.
  (Note that $|M_1| = |M_2| = |M'|$).
  Then  we apply Proposition \ref{amg-prop}. 

\end{itemize}
\end{rm}
\end{example}

\subsection{Model amalgamation by decomposition}

In this part we establish model amalgamation in decomposed stratified
institutions on the basis of the model amalgamation properties of the
components (Proposition \ref{amg-by-decomp}).
Although developed in a much more general theoretical framework this
result goes somehow in the same direction with a corresponding model
amalgamation study from \cite{ks}.
For instance both share the same practical goal of providing an easier 
route for establishing model amalgamation in concrete situations.

The following preliminary result shows how the model amalgamation in
the ``base'' stratified component can be reduced to model amalgamation
in the base institution. 

\begin{proposition}[Model amalgamation in $\wt{\B}$]
  \label{tilde-amg}
  Let $\B$ be any institution.
  Any model amalgamation square in $\B$ is a model amalgamation square
  in $\widetilde{\B}$ too.
\end{proposition}

\begin{proof}
Let the commutative square of signature morphisms below be a model
amalgamation square in $\B$. 
\[
\xymatrix{
  \Sigma \ar[r]^{\varphi_1} \ar[d]_{\varphi_2} & \Sigma_1
  \ar[d]^{\theta_1} \\
  \Sigma_2 \ar[r]_{\theta_2} & \Sigma'
}
\]
Let $(W_k, B_k) \in |\Mod^{\widetilde{\B}} (\Sigma_k)|, k = \ol{1,2}$
such that $\varphi_1 (W_1, B_1) = \varphi_2 (W_2, B_2)$.
This means $W_1 = W_2$ and $\varphi_1 B_1^w = \varphi_2 B_2^w$ for
each $w \in W (= W_1 = W_2)$.
By the model amalgamation hypothesis in $|B$, for each $w\in W$ there
exists an unique $\Sigma'$-model $B'^w$ such that $\theta_k B'^w =
B_k^w, k = \ol{1,2}$.
This gives $B' \co W \to |\Mod^{\B} (\Sigma')|$.
Then $(W,B')$ is the unique amalgamation of $(W,B_1)$ and $(W,B_2)$ in
$\widetilde{\B}$. 
\end{proof}

In the ``base'' component of a decomposition we actually need model
amalgamation at the level of the constrained models.
The following definition provides a general condition that allows for
the model amalgamation established at the level of $\wt{\B}$ in
Proposition \ref{tilde-amg} to be transfered to $\wt{\B}^C$.
The example after Definition \ref{cstr-amg} illustrate how this may
function in concrete situations. 

\begin{definition}\label{cstr-amg}
  Let $\B$ be any institution.
  A constraint model sub-functor $\Mod^C \subseteq
  \Mod^{\widetilde{\B}}$ \emph{preserves amalgamation} when for any
  pushout square of signature morphisms
  \[
  \xymatrix{
  \Sigma \ar[r]^{\varphi_1} \ar[d]_{\varphi_2} & \Sigma_1
  \ar[d]^{\theta_1} \\
  \Sigma_2 \ar[r]_{\theta_2} & \Sigma'
  }
  \]
  and for any $\widetilde{\B}$ $\Sigma'$-model $(W,B')$,
  $\theta_k (W,B') \in |\Mod^C (\Sigma_k)|$, $k = \ol{1,2}$, implies
  $(W,B') \in |\Mod^C (\Sigma')|$. 
\end{definition}

The following transfer of model amalgamation from $\B$ to $\wt{\B}^c$
is an immediate consequence of Proposition \ref{tilde-amg} and of
Defintion \ref{cstr-amg}.

\begin{corollary}
If $\B$ is semi-exact and $\Mod^C \subseteq \Mod^{\wt{\B}}$ preserves
amalgamation then $\wt{\B}^C$ is semi-exact. 
\end{corollary}

\begin{example}\label{share-mfol}
\begin{rm}
$\MPL$, $\HPL$ provide trivial cases for Definition \ref{cstr-amg}
because in both cases $\B = \PL$ and $\Mod^C = \Mod^{\wt{\B}}$.
The situation is different for $\MFOL, \HFOL, \HHPL$, etc.
As an example let us see how is it with $\MFOL$, this case being quite
emblematic for a whole class of examples.
In this case $\B = \FOL$. 

At the level of the underlying carriers the things are simple: since
$W$ is invariant when taking the reducts, i.e. $\theta_k (W,B') =
(W,B_k)$, and the same happens with the underlying carriers, i.e.
$|B'^w| = |B^w_k|$ for each $w \in W$,
\[
  |B'^w|= |B^w_k| = |B^v_k| = |B'^v|.
\]
For the interpretations of the constants let us consider a constant
$c'$ of the $\FOL$ signature $\Sigma'$.
By the pushout property of the square of signature morphisms there
exists $k \in \{ 1,2 \}$ and $c_k$ a constant of $\Sigma_k$ such that
$\theta_k c_k = c'$.
Hence for all $w,v \in W$ we have that
\[
B'^w_{c'} = B'^w_{\theta_k c_k} = (B^w_k)_{c_k} = (B^v_k)_{c_k} =
B'^v_{\theta_k c_k} = B'^v_{c'}.
\]
\end{rm}
\end{example}

The following is the main result of this section. 

\begin{proposition}[Model amalgamation by
  decomposition]\label{amg-by-decomp} 
Consider a decomposition of a stratified institution $\SI$
\[
\xymatrix @C+3em {
\SI^0 & \ar[l]_{(\Phi^0,\alpha^0,\beta^0)} \SI
\ar[r]^{(\Phi,\alpha,\widetilde{\beta})}& \widetilde{\B}^C  
}
\]
such that
\begin{enumerate}

\item $\SI^0$ is strict, 

\item $\Phi$ and $\Phi^0$ preserve pushouts, 

\item $\B$ and $\SI^0$ are semi-exact, and 

\item $\Mod^C$ preserves amalgamation.
  
\end{enumerate}
Then $\SI$ is semi-exact too. 
\end{proposition}

\begin{proof}
Consider a pushout square of signature morphisms in $\SI$
\begin{equation}\label{pushout-square}
\xymatrix{
  \Sigma \ar[r]^{\varphi_1} \ar[d]_{\varphi_2} & \Sigma_1
  \ar[d]^{\theta_1} \\
  \Sigma_2 \ar[r]_{\theta_2} & \Sigma'
}
\end{equation}
and let $M_k \in |\Mod^\SI (\Sigma_k)|$, $k = \ol{1,2}$, such that
$\varphi_1 M_1 = \varphi_2 M_2$.
By relying on the preservation pushout condition we have that the
squares below 
\[
\xymatrix{
  \Sigma^0 \ar[r]^{\varphi^0_1}  \ar[d]_{\varphi^0_2} &
  \Sigma^0_1 \ar[d]^{\theta^0_1} & & \Phi\Sigma \ar[r]^{\Phi\varphi_1}
  \ar[d]_{\Phi\varphi_2} & \Phi\Sigma_1 \ar[d]^{\Phi\theta_1} \\
  \Sigma^0_2 \ar[r]_{\theta^0_2} & \Sigma'^0 & & \Phi\Sigma_2
  \ar[r]_{\Phi\theta_2} & \Phi\Sigma' 
}
\]
are pushout squares in $\Sign^0$ and $\Sign^{\B}$, respectively (where
the left-hand square above is the result of applying $\Phi^0$ to
\eqref{pushout-square}).
We let $M^0_k = \beta^0_{\Sigma_k} M_k$ and $(W_k, N_k) =
\widetilde{\beta}_{\Sigma_k} M_k$, $k = \ol{1,2}$.

Our plan is to obtain the amalgamation of $M_1$ and $M_2$ through the
pullback property on the category of models of the decomposition of
$\SI$ by joining together the amalgamations of $M^0_1$ and $M^0_2$ and
of $(W_1, N_1)$ and $(W_2, N_2)$. 
The first step is to establish the conditions for the two
amalgamations:
\begin{itemize}

\item In the case of $M^0_1$ and $M^0_2$ we have that
  \[
  \varphi^0_k M^0_k = \varphi^0_k (\beta^0_{\Sigma_k} M_k) =
  \beta^0_\Sigma (\varphi_k M_k) \ (\text{by the naturality of } \beta^0)
  \]
  Since by hypothesis $\varphi_1 M_1 = \varphi_2 M_2$ it follows that
  $\varphi^0_1 M^0_1 = \varphi^0_2 M^0_2$. 

\item A similar argument applies also to $(W_1, N_1)$ and $(W_2,
  N_2)$, by using the naturality of $\widetilde{\beta}$ and the
  assumption $\varphi_1 M_1 = \varphi_2 M_2$.  
  
\end{itemize}
Now, by the semi-exactness hypotheses, let $M'^0$ be the unique
amalgamation of $M^0_1$ and $M^0_2$, and $(W',N')$ be the unique
amalgamation of $(W_1, N_1)$ and $(W_2, N_2)$.
Note that in the case of the latter amalgamation we rely upon the
result of Proposition \ref{tilde-amg} and on the preservation of
amalgamation by $\Mod^C$ hypothesis.
Note also that since $\widetilde{\B}^C$ is strict we have that $W_1 =
W_2 = W'$.
We have that
\[
  \begin{array}{rlr}
    \sem{M'^0} & = \sem{(\Phi\theta_1)M'^0} & \SI^0 \ \text{strict} \\
               & = \sem{M^0_1} & \text{definition of} \ M^0_1 \\
               & = \sem{(W_1, N_1)} & \text{decomposition property of}
                                      \ M_1 \\ 
    & = W_1 = W'. & 
  \end{array}
\]
Hence $\sem{M'^0} = \sem{(W',N')}$ which allows us to apply the
pullback property of the model decomposition and define $M'$ to be the
unique $\Sigma'$-model such that $\beta^0_{\Sigma'} M' = M'^0$ and 
$\widetilde{\beta}_{\Sigma'} M' = (W',N')$.

We show that $M'$ is the amalgamation of $M_1$ and $M_2$.
On the one hand, by the naturality of $\beta^0$ and since $M'^0$ has
been defined as the amalgamation of $M^0_1$ and $M^0_2$ we have:
\begin{equation}\label{amg1}
\beta^0_{\Sigma_k} (\theta_k M') = \theta^0_k (\beta^0_{\Sigma'} M')
= \theta^0_k M'^0 = M^0_k. 
\end{equation}
On the other hand, by the naturality of $\widetilde{\beta}$ and since
$(W',N')$ is the amalgamation of $(W_1, N_1)$ and $(W_2, N_2)$ we
have:
\begin{equation}\label{amg2}
\widetilde{\beta}_{\Sigma_k} (\theta_k M') =
(\Phi\theta_k)(\widetilde{\beta}_{\Sigma'} M') = (\Phi\theta_k)(W',N')
= (W_k, N_k). 
\end{equation}
From \eqref{amg1} and \eqref{amg2}, by the uniqueness aspect of the
pullback property of the model decomposition, it follows that
$\theta_k M' = M_k$, $k = \ol{1,2}$.
The uniqueness of $M'$ follows by relying on the uniqueness of the
model amalgamation in both $\SI^0$ and $\widetilde{\B}^C$. 
\end{proof}

In many concrete examples the second condition of Proposition
\ref{amg-by-decomp} is established through corresponding instances of
the following lemma. 

\begin{lemma}\label{prod-signature}
  If the decomposition of the stratified institution has the property
  that 
  \[
    \xymatrix{
    \Sign^0 & \ar[l]_{\Phi^0} \Sign^{\SI}   \ar[r]^{\Phi}& \Sign^\B  
    }
  \]
  is a product in $\CAT$, then both $\Phi^0$ and $\Phi$ preserve
  pushouts.
  Moreover any pair of pushout squares of signatures, one from
  $\SI^0$ and the other one from $\B$, determine canonically a pushout
  square of $\SI$ signatures.
\end{lemma}

\begin{proof}
  By a straightforward general categorical argument.
\end{proof}

The following corollary provides an example of how the result of
Proposition \ref{amg-by-decomp} can be applied in order to obtain
model amalgamation properties in concrete stratified institutions.
It is rather comprehensive with respect to the conditions of
Proposition \ref{amg-by-decomp}.
Other such model amalgamation properties can be established in a
multitude of concrete stratified institutions (such as those given in
Section \ref{ex-sect}) in a similar manner.

\begin{corollary}
$\MMFOL$ is semi-exact. 
\end{corollary}

\begin{proof}
We apply Proposition \ref{amg-by-decomp} be performing a check on its
conditions as follows.
Step 0 consists of recalling from Example \ref{dec3} (see also
Examples \ref{dec1} and \ref{dec2}) the decomposition of $\MMFOL$.
This is
\[
  \xymatrix @C+3em {
\REL & \ar[l]_{(\Phi^0,\alpha^0,\beta^0)} \MMFOL
\ar[r]^{(\Phi,\alpha,\widetilde{\beta})}& \widetilde{\FOL}^C  
}
\]
where
\begin{itemize}

\item $\REL$ is considered as a stratified institution by letting
  $\sem{W} = |W|$, i.e. the underlying set of the $\REL$-model $W$.
  Note that because the $\REL$-signatures have only predicate symbols
  there are no $\REL$ sentences.
  This situation  would be different if instead of $\MMFOL$ we would
  consider $\MHFOL$ (see Example \ref{dec3}). 

\item The constraint model functor $\Mod^C$ that defines $\wt{\FOL}^C$
  is given by the sharing of the underlying sets and of the
  interpretations of the constants (see Examples \ref{poly-ex},
  \ref{mfol-ex}, \ref{share-mfol}). 
  
\end{itemize}
Now we focus on how the four conditions of Proposition
\ref{amg-by-decomp} hold. 
\begin{enumerate}

\item As stratified institution $\REL$ is a strict one because the
  reducts in $\REL$ preserve the underlying sets of the models. 

\item We apply Lemma \ref{prod-signature}.
  The signatures of $\MMFOL$ are indeed pairs $(\Lambda,(F,P))$ where
  $\Lambda$ is a $\REL$ signature and $(F,P)$ is a $\FOL$ signature;
  hence the product condition of Lemma \ref{prod-signature} is
  fulfilled. 

\item In the literature $\FOL$ is a classic example of a semi-exact
  institution (see \cite{modalg,iimt}, etc.) although usually $\FOL$
  is considered in its many-sorted form.
  For our single sorted variant it is just enough to note that the
  pushouts of single sorted $\FOL$ signatures are still single sorted,
  and thus the semi-exactness of single sorted $\FOL$ is inherited
  from the more general many sorted $\FOL$.
  The same argument applies to $\REL$ too, as $\REL$ is a fragment (or
  a sub-institution) of (single sorted) $\FOL$. 

\item This has essentially been established in Example
  \ref{share-mfol}. 
  
\end{enumerate}
\end{proof}

\section{Diagrams in stratified institutions}

In conventional model theory the method of diagrams is one of the most 
important methods.
The institution-independent method of diagrams plays a significant
role in the development of a lot of model theoretic results at the
level of abstract institutions, many of its applications being
presented in \cite{iimt}.
These include existence of co-limits of models, free models along
theory morphisms, axiomatizability results, elementary homomorphisms
results, filtered power embeddings results, saturated models results
(including an abstract version of Keisler-Shelah isomorphism theorem),
the equivalence between initial semantics and quasi-varieties,
Robinson consistency results, interpolation theory, definability
theory, proof systems, predefined types, etc.

In institution theory diagrams had been introduced for the first time
by Tarlecki in \cite{tarlecki86,tarlecki86s} in a form different from
ours. 
In the form presented here it has been introduced at the level of 
institution-independent model theory in \cite{edins} as a 
categorical property which formalizes the idea that
\begin{quote}
  the class of model homomorphisms from a model $M$ can be represented
  (by a natural isomorphism) as a class of models of a theory in a
  signature extending the original signature with syntactic entities
  determined by $M$.
\end{quote}
This can be seen as a coherence property between the
semantic and the syntactic structures of the institution.
By following the basic principle that a structure is rather defined by
its homomorphisms (arrows) than by its objects, the semantic
structure of an institution is given by its model homomorphisms.
On the other hand the syntactic structure of a(ny concrete)
institution is based upon its corresponding concept of atomic
sentence.

The goal of this section is twofold.
On the one hand we need to clarify the concept of diagrams in
\emph{stratified} institutions.
This is quite straightforward:
\begin{quote}
\emph{the diagrams in a stratified institution $\SI$ are the diagrams
  in $\SI^\sharp$.}
\end{quote}
On the other hand it is useful to have a general result on the
\emph{existence} of diagrams at the level of abstract stratified
institutions that would be applicable to a wide class of concrete
situations.
In this section we will develop such a result by reliance on
decompositions of stratified institutions.

The structure of the section is as follows:
\begin{enumerate}

\item We recall the established institution theoretic concept of
  diagrams. 

\item We introduce some preliminary technical concepts that will
  support the development of the main result of this section. 

\item We formulate and prove a general result on the existence of
  diagrams in stratified institutions.
  This comes in two versions: for $\SI^*$ and for $\SI^\sharp$ (where
  $\SI$ is a stratified institution).

\item By means of a (counter)example we show the necessity of the main
  specific technical condition underlying our result on the existence
  of diagrams, namely the specific infrastructure supporting
  nominals. 
  
\end{enumerate}

\subsection{A reminder of institution-theoretic diagrams}

Below we recall from \cite{edins,iimt} the main concept of the
institution theoretic method of diagrams.

\begin{definition}[The method of diagrams]
An institution $\I$ has \emph{diagrams}  when
for each signature $\Sigma$ and each $\Sigma$-model $M$, there exists
a signature $\Sigma_M$ and a signature 
morphism $\iota_\Sigma (M) \co \Sigma \ra \Sigma_M$, functorial in
$\Sigma$ and $M$, and a set $E_M$ of $\Sigma_M$-sentences such that
$\Mod(\Sigma_M, E_M)$ and the comma category $M/\Mod(\Sigma)$ are
naturally
isomorphic, i.e. the following diagram commutes by the isomorphism
$i_{\Sigma,M}$ that is natural in $\Sigma$ and $M$
\snvsp
\begin{equation}\label{diag-diag}
  \xymatrix{
\Mod(\Sigma_M, E_M) \ar[r]^{i_{\Sigma,M}}
\ar[rd]_{\Mod(\iota_{\Sigma}(M))} &
M/\Mod(\Sigma) \ar[d]^{\mbox{\emph{forgetful}}} \\ & \Mod(\Sigma) }
\end{equation}
The signature morphism $\iota_{\Sigma}(M) \co \Sigma \ra \Sigma_M$ is
called the \emph{elementary extension of $\Sigma$ via $M$} and the set
$E_M$ of $\Sigma_M$-sentences is called the
\emph{diagram} of the model $M$.

 
\end{definition}

In the institution theoretic literature, especially in \cite{iimt},
one can find a wealth of examples of systems of diagrams.
Below we remind two of the most common ones. 

\begin{example}[Diagrams in $\PL$]
\begin{rm}
For any $\PL$ signature $P$ and any $P$-model $M \in 2^P$ the
extension $\iota_{P,M}$ is just the identity function on $P$, while
$E_M = M$.
Then, that $N \in 2^P$ satisfied $E_M$ means just that $M \subseteq
N$; this gives $i_{P,M} N$. 
\end{rm}
\end{example}

\begin{example}[Diagrams in $\FOL$]
\begin{rm}
For any $\FOL$ signature $\Sigma = (F,P)$ and any $(F,P)$-model $M$
the extension $\iota_{(F,P),M}$ just adds the set of the elements
$|M|$ of $M$ as new constants to $F$.
The let $M_M$ be the expansion of $M$ along $\iota_{(F,P),M}$ that
interprets the new constants by themselves, i.e. $(M_M)_c = c$ for any
$c \in |M|$.
$E_M$ is defined as the set of the quantifier-free equations
satisfied by $M_M$.
For any $\Sigma_M$ model $N'$ that satisfied $E_M$, $i_{\Sigma,M} N'$
is the $(F,P)$-homomorphism $h \co M \to N$ defined by $h(x) = N'_x$,
where $N$ is the $\iota_{\Sigma,M}$-reduct of $N'$. 
\end{rm}
\end{example}

In order to keep the exposition technically simpler, for the rest of
this section we will ignore the properties of the functoriality of
$\iota$ and of the naturality of $i$ and rather focus on the primary
property of diagrams, i.e. the isomorphism property of $i_{\Sigma,M}$
and the commutativity shown in the diagram \eqref{diag-diag}. 
Moreover, in most applications of institution theoretic diagrams only
this primary property is used.
However the interested reader may develop by himself what the
functoriality and the naturality properties mean in explicit form, or
else he may consult them from \cite{iimt}. 

\subsection{Some supporting technical concepts}

The main idea underlying our development of a general result on the
existence of diagrams in stratified institutions is to consider
decompositions of stratified institutions, to asume diagrams for each
of the two components (which in concrete situations are already known
/ established), and then to combine these at the level of the
stratified institution.
However this process requires some technical conditions that we will
spell out explicitly in what follows.

The first condition supports the lifting of diagrams from $\B$ to
$\B^C$. 

\begin{definition}\label{coherent}
Let $\B$ be an institution and $\Mod^C$ be a constraint model
sub-functor for $\widetilde{\B}$.
A system of diagrams for $\B$ is \emph{coherent with respect to
  $\Mod^C$} when for each $\B$-signature $\Sigma$ and each
$(W,B) \in |\Mod^C (\Sigma)|$  we have that
\begin{enumerate}

\item For all $i,j \in W$, $\iota_{\Sigma,B^i} =
  \iota_{\Sigma,B^j}$; in this case all $\iota_{\Sigma,B^i}$s will be
  denoted by $\iota_{\Sigma,B} \co \Sigma \to \Sigma_B$. 

\item for each $(W,B) \in |\Mod^C (\Sigma)$ there exists a canonical
  isomorphism $i_{\Sigma,(W,B)}$ such that the following diagram
  commutes:
  \[
  \xymatrix{
  \Mod (\Sigma_B, E_{(W,B)}) \ar[rr]^{i_{\Sigma, (W,B)}}_{\cong}
  \ar[dr]_{\Mod^C (\iota_{\Sigma, B})} & &
  (W,B) / \Mod^C (\Sigma) \ar[dl]^{\text{forgetful}} \\
  & \Mod^C (\Sigma) 
  }
\]
where $\Mod (\Sigma_B, E_{(W,B)})$ denotes the subcategory of
the comma category $\sem{M} / \sem{\_}_{\Sigma_B}$ (where
$\sem{\_}_{\Sigma_B} \co \Mod^C (\Sigma_B)\to \Set$)
induced by those objects $(f \co W \to V, (V, N'))$ such that
${N'}^{f(i)} \models E_{B^i}$ for each $i\in W$.\footnote{Note 
  that unlike $E_{B^i}$, $E_{(W,B)}$ is \emph{not} a set of
  sentences.} 
  
\end{enumerate}
\end{definition}

\begin{example}[Coherence in $\widetilde{\PL}$]\label{coh-pl} 
\begin{rm}
  If $\B = \PL$ then $\Mod^C = \Mod^{\widetilde{B}}$.
  The first condition is trivially satisfied because in $\PL$ all
  elementary extensions $\iota_{P,B}$ are identities.
  On this basis the second condition is also trivially satisfied. 
\end{rm}
\end{example}

Although the following example requires a more intricate verification,
this is still rather straightforward. 

\begin{example}[Coherence in $\widetilde{\FOL}$]\label{coh-fol} 
\begin{rm}
If $\B = \FOL$ then the diagrams of $\FOL$ are coherent with respect
to $\Mod^C$ where $(W,B) \in |\Mod^C (F,P)|$ if and only if for all
$i,j \in W$, $|B^i| = |B^j|$ and $B^i_c = B^j_c$ for each $c \in
F_0$.

On the one hand, this is so because for each $\FOL$ $(F,P)$-model
$M$, $\iota_{(F,P),M}$ is the extension of $(F,P)$ with new
constants which are the elements of $M$.

On the other hand, the second condition of Definition \ref{coherent}
goes as follows.
For any $(f \co W \to V, (V,N')) \in \Mod(\Sigma_B, E_{(W,B)})$ we let
$i_{\Sigma,(W,B)} (f,(V,N')) = (f,h)$ where $h \co |B^i| \to
|N^{f(i)}|$ is the function that is invariant with respect to $i\in W$
and which is given by the diagrams of $B^i \co h^i = i_{\Sigma,B^i}
N^{f(i)}$ where $N^{f(i)}$ is the reduct of $N^{f(i)}$ along
$\iota_{\Sigma,B}$ (just forgets the interpretations of the new
constants corresponding to the elements of the models $B^j$, which are
in fact shared by all $B^j$). 
We can talk about \emph{one} function $h$ because as functions $h^i =
h^j$ for all $i,j \in W$.
This is so because for each element $c \in |B^i| = |B^j|$ we have that
$h^i (c) = {N'}^i_c = {N'}^j_c$ (because $(V,N')$ being a constraint
model its components share the interpretations of the constants) $=
h^j_c$.
This also makes $(f,h)$ a constraint model homomorphism, so
$i_{\Sigma, (W,B)} N'$ belongs to $(W,B) / \Mod^C (\Sigma)$ indeed.

The inverse $i_{\Sigma, (W,B)}^{-1}$ is defined as follows.
Given $(f,h) \co (W,B) \to (V,N)$, for each $i\in W$ we let
${N'}^{f(i)}$ be $i_{\Sigma, B^i}^{-1} h$ where $h \co B^i \to
N^{f(i)}$.
This is correctly defined because if $f(i) = f(j)$ then $N^{f(i)} =
N^{f(j)}$ and ${N'^{f(i)}}, {N'^{f(j)}}$ are just the expansions of
${N^{f(i)}}, {N^{f(j)}}$, respectively, with interpretations of new
constants, i.e. ${N'^{f(i)}} = h(c) = {N'^{f(j)}}$.
When $v \in V$ is outside the image of $f$, $N^v$ is uniquely
determined by the constraint as ${N'}^v$ is the expansion of $N^v$
with the interpretations of the elements of $B$ as new constants
which are shared with other ${N'}^i$. 
\end{rm}
\end{example}

The following defines the workable situation when both components of a
decomposition of a stratified institution admit diagrams, this being
the root condition of our approach to the existence of diagrams in
stratified institutions. 

\begin{definition}\label{decomp-diag}
A decomposition of a stratified institution $\SI$
\[
  \xymatrix @C+3em {
  \SI^0 & \ar[l]_{(\Phi^0,\alpha^0,\beta^0)} \SI
  \ar[r]^{(\Phi,\alpha,\widetilde{\beta})}& \widetilde{\B}^C  
  }
\]
\emph{admits diagrams} when ${\SI^0}^*$ and $\B$ have diagrams such
that the diagrams of $\B$ are coherent with respect to $\Mod^C$. 
\end{definition}

\begin{example}[The case of $\HPL$]\label{admits-diag-hpl}
\begin{rm}
  We have to recall Example \ref{dec1}.
  The diagrams of $\SI^0$ are as follows.
  For any $\SI^0$-signature $\Nom$ and any $\Nom$-model $W = (|W|,
  W_\lambda \subseteq |W| \times |W|)$, $\iota_{\Nom,W}$ is the
  extension of $\Nom$ with the elements of the set $|W|$ and
  \[
    E_W = \{ \lambda(i,j) \mid (i,j) \in W_\lambda \}.
  \]
  Since in this example $\B = \PL$ the coherence of the diagrams of
  $\B$ with respect to $\Mod^C$ is explained by Example
  \ref{coh-pl}. 
\end{rm}
\end{example}

\begin{example}[The case of $\HFOL$]\label{admits-diag-hfol}
\begin{rm}
  We have to recall Example \ref{dec2}.
  Since $\SI^0$ is the same like in the $\HPL$ case, the diagrams of
  $\SI^0$ are those given in Example \ref{admits-diag-hpl}. 
  In this example $\B = \FOL$ and therefore the coherence of the
  diagrams of $\B$ with respect to $\Mod^C$ is explained by Example
  \ref{coh-fol}. 
\end{rm}
\end{example}

\begin{notation}\label{iota-not}
\begin{rm}
For any decomposition of a stratified institution that admits
diagrams (like in Definition \ref{decomp-diag})
for any $\Sigma \in |\Sign^\SI|$ and $M \in |\Mod^\SI (\Sigma)|$, we
introduce the following abbreviations: 
\[
  \Sigma_0 = \Phi^0 \Sigma, \ 
  \Sigma_1 = \Phi \Sigma, \
  M_0 = \beta^0_\Sigma M, \
  M_1 = \widetilde{\beta}_\Sigma M. 
\]
We let $\iota_{\Sigma_0, M_0} \co \Sigma_0 \to ({\Sigma_0}_{M_0}, E_{M_0})$
and (for each $i \in \sem{M}$) $\iota_{\Sigma_1, M_1^i} \co \Sigma_1
\to ({\Sigma_1}_{M_1^i}, E_{M_1^i})$ be the diagrams of $M_0$ and
$M_1^i$, respectively.
By the coherence hypothesis we have $\iota_{\Sigma_1, M_1^i} =
\iota_{\Sigma_1, M_1^j}$ for all $i,j \in \sem{M}$.
This allows us to denote all $\iota_{\Sigma_1, M_1^i}$ by
$\iota_{\Sigma_1, M_1}$.

If 
  \[
    \xymatrix{
    \Sign^0 & \ar[l]_{\Phi^0} \Sign^{\SI}   \ar[r]^{\Phi}& \Sign^\B  
    }
  \]
is a product in $\CAT$ then we define the $\Sign^S$ morphism
$\iota_{\Sigma,M} \co \Sigma \to \Sigma_M$ by using the product
property of $(\Phi^0, \Phi)$: 
\[
\iota_{\Sigma,M} = (\iota_{\Sigma_0, M_0}, \iota_{\Sigma_1, M_1}).
\]
\end{rm}
\end{notation}

The last technical concept supporting the main result of this section
expresses the possibility that each element of the underlying
stratifications has a syntactic designation.
Although it has a rather heavy technical appearance it holds naturally
in the examples. 

\begin{definition}\label{nom-def}
Consider a decomposition of a stratified institution that admits
diagrams like in Definition \ref{decomp-diag}.
We say that the diagrams (of the decomposition) \emph{denote the
  stratification} when
\begin{itemize}

\item
  \[
    \xymatrix{
    \Sign^0 & \ar[l]_{\Phi^0} \Sign^{\SI}   \ar[r]^{\Phi}& \Sign^\B  
    }
  \]
is a product in $\CAT$,

\item $\SI$ has a nominals extraction $(N,\Nm)$, 
  
\item for each $\SI$-signature $\Sigma$ and each $\Sigma$-model $M$,
  there exists a function
  \[
    n_{\Sigma,M} \co \sem{M}_\Sigma \to N(\Sigma_M)
  \]
  such that $n$ is natural in $\Sigma$ and $M$, and 

\item for each $\Sigma_M$-model $N$ such that 
$N \models \alpha_0 E_{M_0}$, 
\[
n_{\Sigma,M} ; \Nm_{\Sigma_M} (N) = 
\sem{i_{\Sigma_0, M_0} N_0}_{\Sigma_0}. 
\]
\end{itemize}
\end{definition}

\begin{example}[The case of $\HPL$]\label{nom-hpl-decomp}
\begin{rm}
According to Example \ref{nom-extraction}, $\HPL$ has nominals
extraction $(N,\Nom)$ where for each $\HPL$ signature $(\Nom,P)$,
$N(\Nom,P) = \Nom$ and for each $(\Nom,P)$-model $(W,M)$,
$\Nm(W,M) = (|W|, (W_i)_{i\in \Nom})$.
Then  we define $n_{(\Nm,P),(W,M)}$ as the canonical injection $|W|
\to \Nom + |W|$, where $\Nom + |W|$ denotes the disjoint union of
$\Nom$ and $|W|$.

Now let us consider any $(\Nom,P)_{(W,M)}$-model $(V',N')$ such that
$(V',N') \models \alpha_0 E_W$ which means $V' \models E_W$.
Then for each $w\in |W|$ we have that
\[
(\Nm (V',N'))_{n(w)} = (|V|, (V'_i)_{i\in \Nom + |W|})_w = V'_w =
(i_{\Nom,W} V')w = \sem{i_{\Nom,W} V'} w.
\]
\end{rm}
\end{example}

\begin{example}[The case of $\HFOL$]\label{nom-hfol-decomp}
\begin{rm}
This is similar to Example \ref{nom-hpl-decomp} because the property
of Definition \ref{nom-def} depends essentially on the $\SI^0$ part of
the decomposition of $\SI$, which is shared between $\HPL$ and $\HFOL$. 
\end{rm}
\end{example}

\subsection{The existence of diagrams in stratified institutions}

\begin{theorem}\label{diag-thm}
For any decomposition of a stratified institution $\SI$ that admits
diagrams that denote the stratification:
\begin{itemize}

\item $\SI^*$ has diagrams when $\SI$ has explicit local satisfaction, 
  and 

\item $\SI^\sharp$ has diagrams when $\SI$ has explicit local
  satisfaction and has $i$-sentences too.
  
\end{itemize}
\end{theorem}

\begin{proof}
For each $\SI$ signature $\Sigma$ and each $\Sigma$-model $M$ we
define $E_M \subseteq \Sen^{\SI} (\Sigma_M)$ by 
\[
E_M = \alpha^0_{\Sigma_M} E_{M_0} \ \cup \ \bigcup_{i\in \sem{M}} @_i
(\alpha_{\Sigma_M} E_{M_1^i})
\]
where $@_i (\alpha_{\Sigma_M} E_{M_1^i})$ abbreviates
$\{ @_{n_{\Sigma,M} (i)} \alpha_{\Sigma_M} \rho \mid \rho \in
E_{M_1^i} \}$. 

We will prove that $\iota_{\Sigma,M} \co \Sigma \to (\Sigma_M, E_M)$
(see Notation \ref{iota-not}) is the diagram of $M$ in $\SI^*$.
The coherence condition implies that there exists a canonical 
isomorphism $i_{\Sigma_1, M_1}$ such that the following diagram
commutes:
\[
  \xymatrix{\
  \Mod ({\Sigma_1}_{M_1}, E_{M_1}) \ar[rr]^{i_{\Sigma_1, M_1}}_{\cong}
  \ar[dr]_{\Mod^C (\iota_{\Sigma_1, M_1})} & &
  M_1 / \Mod^C (\Sigma_1) \ar[dl]^{\text{forgetful}} \\
  & \Mod^C (\Sigma_1) 
  }
\]

Let $\gamma_{\Sigma,M} \co \Mod(\Sigma_M) \to \Mod({\Sigma_1}_{M_1})$ be
the functor defined by
\[
\gamma_{\Sigma,M} N' = (f\co \sem{M} \to \sem{N'},
\widetilde{\beta}_{\Sigma_M} N')
\]
where $f(i) = \Nm_{\Sigma_M} (N')_{n_{\Sigma,M} i}$.
Then the restriction of $\gamma_{\Sigma,M}$ to $\Mod(\Sigma_M, E_M)$
yields a functor
$\Mod(\Sigma_M, E_M) \to \Mod({\Sigma_1}_{M_1}, E_{M_1})$, as follows.

Let $N'$ be a $\Sigma_M$-model such that $N' \models E_M$.
Then we have that
\begin{proofsteps}[stepsep=-2pt, justwidth=9em]
  \step[g1]{$N' \models_{\Sigma_M} @_i (\alpha_{\Sigma_M} E_{M_1^i})$
    for each $i\in \sem{M}$}{}
  \step[g2]{$N' \models_{\Sigma_M} @_{n_{\Sigma,M}i} (\alpha_{\Sigma_M}
    E_{M_1^i})$ for each $i\in \sem{M}$}{\ref{g1} rewritten}
  \step[g3]{$N' \models^{\Nm_{\Sigma,M} (N')_{n_{\Sigma,M}i}} \alpha_{\Sigma_M}
    E_{M_1^i}$ for each $i\in \sem{M}$}{from \ref{g2} by definition of
    $@_i$}   
  \step[g4]{$\widetilde{\beta}_{\Sigma_M} N' \models^{\Nm_{\Sigma,M}
      (N')_{n_{\Sigma,M}i}} E_{M_1^i}$ for each $i\in \sem{M}$}{from
    \ref{g3} by the Satisfaction Condition of
    $(\Phi,\alpha,\widetilde{\beta})$ }
  \step[g5]{$(\lambda i . \Nm_{\Sigma,M} (N')_{n_{\Sigma,M}i},
    \widetilde{\beta}_{\Sigma_M} N') \in \Mod({\Sigma_1}_{M_1},
    E_{M_1})$}{from \ref{g4} by definition of
    $\Mod({\Sigma_1}_{M_1},E_{M_1})$}
  \step[g6]{$\gamma_{\Sigma,M} N' \in \Mod({\Sigma_1}_{M_1},
    E_{M_1})$}{from \ref{g5} by the definition of $\gamma$}
\end{proofsteps} 

In the diagram below the upper left square is a pullback square.
This follows by general categorical considerations on the basis of the
pullback condition on model categories from the decomposition of
$\SI$.  
\begin{equation}\label{s-diag}
  \xymatrix{
    \sem{M} / \Set & M_0 / \Mod^0 (\Sigma_0)
    \ar[l]_{\sem{\_}_{\Sigma_0}}  & \\
    M_1 /\Mod(\Sigma_1) \ar[u]^{\sem{\_}_{\Sigma_1}} &
    M / \Mod^\SI (\Sigma) \ar[ul]_{\sem{\_}_\Sigma}
    \ar[u]^{\beta^0_\Sigma}
    \ar[l]_{\widetilde{\beta}_\Sigma} &
    \Mod^0 ({\Sigma_0}_{M_0}, E_{M_0})
    \ar[ul]_{i_{\Sigma_0, M_0}}^\cong \\
    & \Mod({\Sigma_1}_{M_1}, E_{M_1}) \ar[ul]^{i_{\Sigma_1, M_1}}_\cong &
    \Mod(\Sigma_M, E_M) \ar[l]^{\gamma_{\Sigma,M}}
    \ar@{..>}[ul]_{i_{\Sigma,M}}^\cong \ar[u]_{\beta^0_{\Sigma_M}}
  }
\end{equation}
If we proved that the outer hexagon of the above diagram represents a
pullback too, then we obtain the isomorphism
$i_{\Sigma,M} \co \Mod^\SI (\Sigma_M, E_M) \to M / \Mod^\SI (\Sigma)$.

We first show the commutativity of the outer hexagon.
For each $\Sigma_M$-model $N'$ such that $N' \models E_M$ we have that

\

\noindent
$\sem{i_{\Sigma_1, M_1} (\gamma_{\Sigma,M} N')}_{\Sigma_1} = $ 
\begin{proofsteps}[stepsep=-2pt, justwidth=11em]
  \step*{}{$= \sem{i_{\Sigma_1, M_1} (\lambda i . \Nm_{\Sigma_M}
      (N')_{n_{\Sigma,M} (i)}, \widetilde{\beta}_{\Sigma_M}
      N')}_{\Sigma_1}$}{definition of $\gamma$} 
  \step*{}{$= \lambda i . \Nm_{\Sigma_M} (N')_{n_{\Sigma,M}
      (i)}$}{definition of $i_{\Sigma_1, M_1}$}
  \step*{}{$= \sem{i_{\Sigma_0, M_0} (\beta^0_{\Sigma_M}
      N')}_{\Sigma_0}$}{Definition \ref{nom-def}.}   
\end{proofsteps}

Finally, we show that the hexagon represents a pullback.
We must prove that given any $(f,N_1) \in \Mod({\Sigma_1}_{M_1},
E_{M_1})$ and $N_0 \in |\Mod^0 ({\Sigma_0}_{M_0}, E_{M_0})|$ such that
$\sem{i_{\Sigma_0, M_0} N_0} = f$ there exists an unique
$N \in \Mod(\Sigma_M, E_M)$ such that $\gamma_{\Sigma,M} N = (f, N_1)$
and $\beta^0_{\Sigma_M} N = N_0$.
It follows that $ N_1 = \widetilde{\beta}_{\Sigma_M} N$.
Note that from $\sem{i_{\Sigma_0, M_0} N_0} = f$ it also follows that
$\sem{N_0} = \sem{N_1}$.
Hence by the pullback of the categories of models of the decomposition 
there exists an unique $N$ such that $N_1 =
\widetilde{\beta}_{\Sigma_M} N$ and $N_0 = \beta^0_{\Sigma_M} N$.
Moreover $f$ is uniquely determined by the condition
$\sem{i_{\Sigma_0, M_0} N_0} = f$. 

For the second conclusion of the theorem, for each $\SI$-signature
$\Sigma$, each $\Sigma$-model $M$, and each $w\in \sem{M}_\Sigma$, let
us abbreviate $\sen{n_{\Sigma,M} (w)}$ by $\sen{w}$.
Let $\iota_{\Sigma,M} \co \Sigma \to (\Sigma_M, E_M)$ be the diagram
corresponding to $M$ in $\SI^*$ as established above.
The we prove that for each $\SI^\sharp$ $\Sigma$-model $(M,w)$,
$\iota_{\Sigma,M} \co \Sigma \to (\Sigma_M, E_M)$ is a diagram for
$(M,w)$ where
\[
E_{(M,w)} = E_M \cup \{ \sen{w} \}.
\]
We prove that the isomorphism $i_{\Sigma,M}$ (in $\SI^*$) can be
extended to an isomorphism
\[
i_{\Sigma,(M,w)} \co \Mod^\sharp (\Sigma_M, E_{(M,w)}) \to (M,w) /
\Mod^\sharp (\Sigma). 
\]
For any $\SI^\sharp$ $\Sigma_M$-model $(N',v)$ that satisfies
$E_{(M,w)}$ we have that $N' \models E_M$.
Let $h = i_{\Sigma,M} N' \co M \to N$.
It remains to show that $(h,w) \co (M,w) \to (N,v)$ is a homomorphism
in $\SI^\sharp$ is equivalent to $(N',v) \models \sen{w}$:

\

\noindent
$(N',v) \models \sen{w}$ 
\begin{proofsteps}[stepsep=-2pt, justwidth=18em]
  \step*{}{$\lc{\equiv} \ \Nm_{\Sigma_M} (N')_{n_{\Sigma,M} w} =
    v$}{definition of satisfaction of $i$-sentences} 
  \step*{}{$\lc{\equiv} \ \sem{i_{\Sigma_0, M_0} N'_0}_{\Sigma_0} w =
    v$}{by Definition \ref{nom-def}}
  \step*{}{$\lc{\equiv} \ \sem{i_{\Sigma, M} N'}_{\Sigma} w =
    v$}{by the commutativity of the upper right half of
    \eqref{s-diag}}
  \step*{}{$\lc{\equiv} \ h(w) = v$.}{}   
\end{proofsteps}

\end{proof}

We can apply Theorem \ref{diag-thm} for $\SI = \HPL$ and $\SI =
\HFOL$ and obtain the following two corollaries, which are emblematic
for applications of this general result. 

\begin{corollary}\label{hpl-diag-prop}
  $\HPL^*$ and $\HPL^\sharp$ have diagrams.
\end{corollary}

\begin{proof}
  We have to recall Example \ref{admits-diag-hpl}.
  For each $(\Nom,P)$-model $(W,M)$: 
\begin{itemize}

\item $\iota_{(\Nom,P),(W,M)}$ is the signature extension with
  nominals $(\Nom,P) \ra (\Nom + |W|,P)$; and

\item $E_{(\Nom,P),(W,M)} = \{ @_i \Diamond j \mid (i,j)\in W_\lambda
  \} \cup \{ @_i \pi \mid \pi \in M^i, i\in |W| \}$.

\end{itemize}
Consequently $\HPL^\sharp$ has diagrams that are defined for
each model $((W,M),w)$ as follows:
\begin{itemize}

\item the elementary extensions are the same as for the $\HPL$
  diagrams; and

\item $E_{(\Nom,P),((W,M),w)} = E_{(\Nom,P),(W,M)} \cup \{   \sen{w} \}$.

\end{itemize}
\end{proof}

\begin{corollary}\label{qhl-diag-prop}
  $\HFOL^*$ and $\HFOL^\sharp$ have diagrams.
\end{corollary}

\begin{proof}
  We have to recall Example \ref{admits-diag-hfol}.
  For each $(\Nom,F,P)$-model $(W,M)$:

  \begin{itemize}

\item $\iota_{(\Nom,F,P),(W,M)}$ is the signature extension to
  $(\Nom+|W|,F+|M|,P)$; and 

\item $E_{(\Nom,F,P),(W,M)} = \{ @_i \Diamond j \mid (i,j)\in W_\lambda
  \} \cup \{ @_i \rho \mid \rho \in E_{(F,P),M^i}, i\in |W| \}$, where
  $E_{(F,P),M^i}$ denotes the $\FOL$ diagram of $M^i$. 

\end{itemize}
$\HFOL^\sharp$ has diagrams that are defined for each $((W,M),w)$ as
follows: 
\begin{itemize}

\item the elementary extensions are the same as for the $\HFOL$
  diagrams; and

\item $E_{(\Nom,F,P),((W,M),w)} =
  E_{(\Nom,F,P),(W,M)} \cup \{   \sen{w} \}$.

\end{itemize}
\end{proof}

\subsection{Non-existence of diagrams}

One of the general conclusions of our study of diagrams for stratified
institutions is that they are dependent on some kind of hybrid
infrastructure.
While mathematically Theorem \ref{diag-thm} only says that such
infrastructure is sufficient, this also \emph{feels} necessary because
the whole idea of diagrams is related to having syntactic designations
for all elements of the models, either ``worlds'' or elements of their
interpretations. 
The following negative result\footnote{Developed jointly with
  Manuel-Antonio Martins.} provides some support for this
conclusion.

\begin{proposition}
Neither $\MPL^*$ nor $\MPL^\sharp$ admit institution theoretic
diagrams. 
\end{proposition}

\begin{proof}
  In both cases we perform a \emph{Reductio ad Absurdum} proof by
  initially assuming that each of the two institutions admit
  diagrams.

  \begin{itemize}

  \item 
  In the case of $\MPL^*$, let $(W,M)$ be any $P$-model in $\MPL$ such
  that $|W|\not= \emptyset$. 
Since $1_{(W,M)}$ is initial in the comma category
$(W,M)/\Mod^{\MPL}(P)$ it follows that $i_{P,(W,M)}^{-1} (1_{(W,M)})$
is initial in the category of the $P_{(W,M)}$-models satisfying
$E_{P,(W,M)}$. 
But it is easy to note that the empty $\MPL$ Kripke structure
trivially satisfies any sentence in $\MPL$ hence $i_{P,(W,M)}^{-1}
(1_{(W,M)})$ is bound to be this trivial empty Kripke structure.  
This is a contradiction because, according to the axioms of diagrams,
when reducing $i_{P,(W,M)}^{-1} (1_{(W,M)})$ via $\iota_{P,(W,M)}$ we
should obtain $(W,M)$ which by our assumption is not empty. 

\item 
For the case of $\MPL^\sharp$ let us consider a singleton signature $P
= \{ \pi \}$ and the $\MPL^\sharp$ $P$-model $((W,M),w)$ where 
\begin{itemize}

\item $|W| = \{ w,v \}$ and $W_\lambda = \emptyset$; and 

\item $M^w = \emptyset$ and $M^v = P$; 

\end{itemize}
Since $P$ is a singleton, without any loss of generality we may assume
that the elementary extension $\iota_{P,((W,M),w)}$ is an inclusion 
$P \subseteq P'$.   
Let 
\[
((W,M'),w) = i_{P,((W,M),w)}^{-1} (1_{(W,M),w)}) 
\]
and let the $\MPL^\sharp$ $P'$-model $((W,N'),w)$ be defined by $N'^w
= M'^w$ and $N'^v = M'^v \setminus P$.   
By induction on the structure of $\rho$, it is easy to establish that
for any $P'$-sentence $\rho$ we have that 
\[
((W,M'),w) \models \rho \mbox{ \ if and only if \ }
((W,N'),w) \models \rho.
\]
Since $((W,M'),w)$ is a model of the diagram of $((W,M),w)$, it
follows that $((W,N'),w)$ is a model of that diagram too. 
Hence $i_{P,((W,M),w)}((W,N'),w))$ is a homomorphism 
$((W,M),w) \ra ((W,N),w)$ where $((W,N),w)$ is the $P$-reduct of
$((W,N'),w)$. 
Let us denote it by $h$. 
Then by the homomorphism property of $h$ we have that 
$M^v \subseteq N^{h(v)}$. 
But $M^v = P$ and $N^w = N^v = \emptyset$, hence $M^v \not\subseteq
N^{h(v)}$. 
This contradiction invalidates our supposition of the existence of
diagrams.

\end{itemize}
\end{proof}

\section{Conclusions}

In this paper we have introduced a new technique for representing
stratified institutions by a decomposition at the level of the
models.
Then we applied this decomposition technique for developing general
results on the existence of model amalgamation and of diagrams in
stratified institutions.
In the latter case it has emerged that some nominals infrastructure is
needed.
This is hardly surprising because fundamentally diagrams reflect a
fine balance between syntax and semantics (i.e. model homomorphisms
represented as models of theories) and the presence of nominals
restore such a balance for Kripke semnantics.

\subsection*{Future work}
The potential of our decomposition technique should be further
explored along the following directions:
\begin{enumerate}

\item Quasi-varieties and initial semantics in stratified
  institutions.
  These have been studied for the particular half-abstract case of
  hybridised institutions in \cite{QVHybrid}.
  However by the decomposition technique we should be able to do those
  at the higher level of generality of abstract stratified
  institutions, one of the consequences being a wider class of
  concrete applications.

\item Develop some general results supporting the existence of
  important structures in stratified institutional model theory that
  in the current literature have an ``assumed'' status.
  An example is that of filtered products of models
  \cite{KripkeStrat}. 

\item Generate new interesting examples of stratified institutions
  that break from the modal logics tradition.
  In this respect pragmatic motivations may come from computing
  science which has many areas whose foundations involve some form  of
  models with states.
  In those situations there is usually an almost automatic reliance on
  modal logics in their more or less conventional acceptations,
  although those have not been developed for those specific purposes,
  but rather for pure logic interests.
  Stratified institutions and their decomposition technique has the
  potential to offer a powerful theoretical tool for going beyond
  modal logics by defining model theoretic frameworks that are finer
  tuned to respective concrete applications. 
  
\end{enumerate}

\subsection*{Acknowledgement}
This work was supported by a grant of the Romanian Ministry of
Education and Research, CNCS -- UEFISCDI, project number
PN-III-P4-ID-PCE-2020-0446, within PNCDI III.

\bibliographystyle{plain}
\bibliography{/Users/diacon/TEX/tex}%

\end{document}